\documentclass[preprint]{elsarticle}
\usepackage{latexsym,amsfonts,amssymb,amsmath,amsthm,amscd}

\theoremstyle{plain}
\newtheorem{theorem}{Theorem}[section]
\newtheorem{lemma}{Lemma}[section]
\newtheorem{proposition}{Proposition}[section]

\newtheorem{definition}{Definition}[section]
\newtheorem{example}{Example}[section]

\newtheorem{remark}{Remark}[section]

\numberwithin{equation}{section}
\numberwithin{definition}{section}

\newcommand{\JM}{Mierczy\'nski}
\newcommand{\RR}{\ensuremath{\mathbb{R}}}

\newcommand{\ZZ}{\ensuremath{\mathbb{Z}}}

\newcommand{\PP}{\ensuremath{\mathbb{P}}}
\newcommand{\TT}{\ensuremath{\mathbb{T}}}
\newcommand{\OFP}{\ensuremath{(\Omega,\mathfrak{F},\PP)}}

\newcommand{\abs}[1]{\ensuremath{\lvert#1\rvert}}
\newcommand{\intpart}[1]{\ensuremath{\lfloor#1\rfloor}}
\newcommand{\norm}[1]{\ensuremath{\lVert#1\rVert}}
\newcommand{\normone}[1]{\ensuremath{\lVert#1\rVert_{1}}}

\newcommand{\mlsps}{measurable linear skew\nobreakdash-\hspace{0pt}product
semidynamical system}

\DeclareMathOperator{\cl}{cl}

\DeclareMathOperator{\Image}{Im}

\DeclareMathOperator{\lnplus}{ln^{+}}

\DeclareMathOperator{\spanned}{span}

\begin{document}
\begin{frontmatter}

\title{Principal Lyapunov Exponents and Principal Floquet Spaces of
Positive Random Dynamical Systems. II. Finite-dimensional Systems}

\author[wroclaw]{Janusz Mierczy\'nski\corref{cor1}\fnref{fn1}}
\ead{mierczyn@pwr.wroc.pl}

\author[auburn]{Wenxian Shen\fnref{fn2}} \ead{wenxish@auburn.edu}

\cortext[cor1]{Corresponding author}

\fntext[fn1]{Supported from resources for science in years 2009-2012
as research project (grant MENII N N201 394537, Poland)}

\fntext[fn2]{Partially supported by NSF grant DMS-0907752}

\address[wroclaw]{Institute of Mathematics and Computer Science,
Wroc{\l}aw University of Technology, Wybrze\.ze Wyspia\'nskiego 27,
PL-50-370 Wroc{\l}aw, Poland}

\address[auburn]{Department of Mathematics and Statistics, Auburn
University, Auburn University, AL 36849, USA}

\begin{abstract}
This is the second part in a series of papers concerned with
principal Lyapunov exponents and principal Floquet subspaces of
positive random dynamical systems in ordered Banach spaces.  The
current part focuses on applications of general theory, developed in
the authors' paper \textit{Principal Lyapunov exponents and principal
Floquet spaces of positive random dynamical systems. I. General
theory}, Trans.\ Amer.\ Math.\ Soc., in~press, to positive random
dynamical systems on finite\nobreakdash-\hspace{0pt}dimensional
ordered Banach spaces.  It is shown under some quite general
assumptions that measurable linear
skew\nobreakdash-\hspace{0pt}product semidynamical systems generated
by measurable families of positive matrices and by strongly
cooperative or type-$K$ strongly monotone systems of linear ordinary
differential equations admit measurable families of generalized
principal Floquet subspaces, generalized principal Lyapunov
exponents, and generalized exponential separations.
\end{abstract}

\end{frontmatter}

\section{Introduction}
\label{section-introduction}

This is the second part of a series of several papers. The series is
devoted to the study of  principal Lyapunov exponents and principal
Floquet subspaces of positive random dynamical systems in  ordered
Banach spaces.

Lyapunov exponents play an important role in the study of asymptotic
dynamics of linear and nonlinear random evolution systems.  Oseledets
obtained in~\cite{Ose} important results on Lyapunov exponents and
measurable invariant families of subspaces for
finite\nobreakdash-\hspace{0pt}dimensional dynamical systems, which
are called now the Oseledets multiplicative ergodic theorem.  Since
then a huge amount of research has been carried out  toward
alternative proofs of the Oseledets multiplicative ergodic theorem
(see \cite{Arn}, \cite{JoPaSe}, \cite{Kre}, \cite{Man}, \cite{Mil},
\cite{Rag}, \cite{Rue1} and the references contained therein) and
extensions of the Osedelets multiplicative theorem for finite
dimensional systems to certain infinite dimensional ones (see
\cite{Arn}, \cite{JoPaSe}, \cite{Kre}, \cite{Lian-Lu},  \cite{Man},
\cite{Mil}, \cite{Rag}, \cite{Rue1}, \cite{Rue2}, \cite{ScFl}, and
references therein).

The largest finite Lyapunov exponents (or top Lyapunov exponents) and
the associated invariant subspaces of both deterministic and random
dynamical systems play special roles in the applications to nonlinear
systems. Classically, the top finite Lyapunov exponent of a positive
deterministic or random dynamical system in an ordered Banach space
is called the {\em principal Lyapunov exponent\/} if the associated
invariant family of subspaces corresponding to it consists of
one\nobreakdash-\hspace{0pt}dimensional subspaces spanned by a
positive vector (in such case, invariant subspaces are called the
{\em principal Floquet subspaces\/}).  For more on those subjects
see~\cite{MiShPart1}.

In the first part of the series, \cite{MiShPart1}, we introduced the
notions of generalized principal Floquet subspaces, generalized
principal Lyapunov exponents, and generalized exponential
separations, which extend the corresponding classical notions. The
classical theory of principal Lyapunov exponents, principal Floquet
subspaces, and exponential separations for strongly positive and
compact deterministic systems is extended to quite general positive
random dynamical systems in ordered Banach spaces.

In the present, second part of the series, we consider applications
of the general theory developed in \cite{MiShPart1} to positive
random dynamical systems arising from a variety of random mappings
and ordinary differential equations.  To be more specific, let
$(\OFP,\theta_t)$ be an ergodic metric dynamical system.  We
investigate positive random matrix models of the form
$((U_{\omega}(n))_{\omega \in \Omega, n \in \ZZ^+}, (\theta_{n})_{n
\in \ZZ})$ (including random Leslie matrix models) (see
Section~\ref{section-matrix}), where
\begin{equation}
\label{matrix-eq0}
U_{\omega}(1) u =
\left(\begin{matrix}
s_{11}(\omega)&s_{12}(\omega)&\cdots&s_{1N}(\omega) \\
s_{21}(\omega)&s_{22}(\omega)&\cdots&s_{2N}(\omega) \\
\vdots & \vdots & \ddots& \vdots \\
s_{N1}(\omega)&s_{N2}(\omega)&\cdots&s_{NN}(\omega)
\end{matrix}\right)u, \quad u \in \RR^N,
\end{equation}
$s_{ij}(\omega) \ge 0$ for $i, j = 1, 2, \dots, N$ and $\omega \in
\Omega$; random cooperative systems of ordinary differential
equations of the form (see Subsection~\ref{subsection-cooperative})
\begin{equation}
\label{cooperative-eq0}
\dot{u}(t) = A(\theta_t\omega)u(t), \qquad \omega \in \Omega, \ t \in
\RR, \ u \in \RR^{N},
\end{equation}
where
\begin{equation*}
A(\omega) =
\left(
\begin{matrix}
a_{11}(\omega)&a_{12}(\omega)&\cdots&a_{1N}(\omega) \\
a_{21}(\omega)&a_{22}(\omega)&\cdots&a_{2N}(\omega) \\
\vdots & \vdots & \ddots & \vdots \\
a_{N1}(\omega)&a_{N2}(\omega)&\cdots&a_{NN}(\omega)
\end{matrix}\right),
\end{equation*}
and $a_{ij}(\omega) \ge 0$  for $i \ne j$, $i, j = 1, 2, \dots, N$
and $\omega \in \Omega$; and random type-$K$ monotone systems of
ordinary differential equations (see
Subsection~\ref{subsection-competitive})
\begin{equation}
\label{k-competitive-eq0}
\dot{u}(t) = B(\theta_t\omega)u(t), \qquad \omega \in \Omega, \ t \in
\RR, \ u \in \RR^{N},
\end{equation}
where for each $\omega\in\Omega$,
\begin{equation*}
B(\omega) =
\left(
\begin{matrix}
b_{11}(\omega)&b_{12}(\omega)&\cdots&b_{1N}(\omega) \\
b_{21}(\omega)&b_{22}(\omega)&\cdots&b_{2N}(\omega) \\
\vdots & \vdots & \ddots & \vdots \\
b_{N1}(\omega)&b_{N2}(\omega)&\cdots&b_{NN}(\omega)
\end{matrix}\right),
\end{equation*}
and there are $1 \le k, l \le N$ such that $k + l = N$,
$b_{ij}(\omega) \ge 0$ for $i\not = j$ and $i, j \in \{1,2,\dots,k\}$
or $i, j \in \{k+1,k+2, \dots,k+l\}$, and $b_{ij}(\omega) \le 0$ for
$i \in \{1, \dots, k\}$ and $j \in \{k+1, \dots, k+l\}$ or $i \in
\{k+1, \dots, k+l\}$ and $j \in \{1,\cdots,k\}$.

We remark that, biologically, \eqref{matrix-eq0} describes
discrete-time age-structured population models,
\eqref{cooperative-eq0} models a system of $N$ species which is
cooperative, and \eqref{k-competitive-eq0} models a community of $N$
species which can be divided into two subcommunities, one
subcommunity consisting of $k$ species and the other consisting of
$l$ species, such that the interactions between every pair of species
in either subcommunity are cooperative, whereas the interactions
between the species belonging to different subcommunities are
competitive. The study of \eqref{matrix-eq0},
\eqref{cooperative-eq0}, and \eqref{k-competitive-eq0} will provide
some basic tool for the study of random discrete-time age-structured
nonlinear population models and random cooperative or type-$K$
monotone systems of nonlinear  ordinary differential equations. The
reader is referred to \cite{BS}, \cite{GyWa}, \cite{Hir},
\cite{HirSmi}, \cite{KoIw}, \cite{LiJi}, \cite{Smi1}, \cite{Smi2},
\cite{Zha}, and references therein for the study of discrete-time
age-structured population models and time periodic cooperative and
type-$K$ monotone systems of nonlinear ordinary differential
equations.

Under quite general conditions, \eqref{matrix-eq0},
\eqref{cooperative-eq0}, generate measurable linear skew-product
semidynamical systems on $\Omega \times \RR^{N}$, preserving the
natural ordering on $\RR^N$ (i.e., the order generated by the cone
$(\RR^N)^+ = \{\, (u = (u_1,\dots,u_N)^{\top} : u_i \ge 0, i =
1,\dots, N\,\}$), and \eqref{k-competitive-eq0} generates a
measurable linear skew\nobreakdash-\hspace{0pt}product semidynamical
system on $\Omega \times \RR^{N}$, preserving the type-$K$ ordering
on $\RR^N$ generated by $(\RR^k)^{+} \times (\RR^l)^{-} = \{\, u =
(u_1,\dots,u_N)^{\top} : u_i \ge 0$ for $i = 1,\dots,k$ and $u_i \le
0$ for $i = k+1,\dots,k+l(= N)\,\}$.

Observe that by the following variable change:
\begin{equation*}
u_i \mapsto u_i \text{ for } i = 1,\dots, k \text{ and } u_i \mapsto
-u_i \text{ for } i = k+1, \dots, k+l (=N),
\end{equation*}
the random type-$K$ monotone system \eqref{k-competitive-eq0} becomes
a random cooperative system of form \eqref{cooperative-eq0} (see
Subsection \ref{subsection-competitive} for detail). We will
therefore focus on the study of \eqref{matrix-eq0} and
\eqref{cooperative-eq0}. Applying the general theory developed in
Part I (\cite{MiShPart1}), we obtain the following results.
\begin{itemize}
\item[(1)]
Under some general positivity assumptions, \eqref{matrix-eq0},
\eqref{cooperative-eq0}, have nontrivial entire positive orbits
(see Theorems \ref{matrix-thm}(1), \ref{cooperative-thm}(1) for
detail);
\item[(2)]
Assume the general positivity and some focusing property.
\eqref{matrix-eq0}, \eqref{cooperative-eq0} have measurable
invariant families of one\nobreakdash-\hspace{0pt}dimensional
subspaces $\{\tilde E_1(\omega)\}$ spanned by positive vectors
({\em generalized principal Floquet subspaces}) and whose
associated Lyapunov exponent is the top Lyapunov exponent of the
system ({\em generalized principal Lyapunov exponent}) (see
Theorems \ref{matrix-thm}(2), \ref{cooperative-thm}(2) for
detail);
\item[(3)]
Assume the general positivity and some strong focusing property.
\eqref{matrix-eq0}, \eqref{cooperative-eq0} have also measurable
invariant families of one\nobreakdash-\hspace{0pt}codimensional
subspaces which are exponentially separated from the generalized
principal Floquet subspaces (see Theorems \ref{matrix-thm}(3),
\ref{cooperative-thm}(3) for detail);
\item[(4)]
Assume the general positivity with some strong positivity in one
direction and some strong focusing property.  Then the
generalized principal Lyapunov exponent for \eqref{matrix-eq0},
\eqref{cooperative-eq0} is finite, so those equations admit
principal Floquet subspaces, principal Lyapunov exponents, and
exponential separation in the classical sense (see Theorems
\ref{matrix-thm}(4), \ref{cooperative-thm}(4) and for detail).
\end{itemize}

The results (1)--(3) are new.  The result (4) recovers many existing
results on principal Floquet subspaces and principal Lyapunov
exponents for \eqref{matrix-eq0}, \eqref{cooperative-eq0} known in
the literature (see \cite{Arn}, \cite{AGD}, \cite{MiSh3}, etc.)

We remark that the generalized principal Lyapunov exponents in (2)
may be $-\infty$.  In such a case, when generalized exponential
separation holds, the (nontrivial) invariant measurable decomposition
associated with the generalized exponential separation is essentially
finer than the (trivial) decomposition in the Oseledets
multiplicative ergodic theorem (see
Subsection~\ref{subsection:example}).

The results obtained in this paper have important biological
implications. For example, the result (3) implies that the population
densities of all species in a random cooperative system of $N$
species increase or decrease at the same rate which equals the
generalized principal eigenvalue of the system along the same
direction which is the direction of the generalized principal Floquet
subspace (see Remarks \ref{grow-along-principal-direction-rk1} and
\ref{grow-along-principal-direction-rk2}).

\smallskip
The rest of this paper is organized as follows. First, for the
reader's convenience, in Section~\ref{general-theory} we put the
notions, assumptions, definitions, and main results of Part I
(\cite{MiShPart1}) in the context of
finite\nobreakdash-\hspace{0pt}dimensional systems.  We then consider
random systems arising from random families of matrices and
cooperative and type-$K$ monotone systems of ordinary differential
equations in Sections~\ref{section-matrix} and
\ref{section-cooperative}, respectively.

\section{General Theory}
\label{general-theory}

\subsection{Notions, Assumptions, and Definitions}
In this subsection, we introduce the notions, assumptions, and
definitions introduced in Part I. The reader is referred to Part I
(\cite{MiShPart1}) for detail.

First, we introduce some notions.

For a metric space $Y$, $\mathfrak{B}(Y)$ stands for the
$\sigma$\nobreakdash-\hspace{0pt}algebra of all Borel subsets of $Y$.
A pair $(\Omega, \mathfrak{F})$ is called a {\em measurable space\/}
if $\Omega$ is a set and $\mathfrak{F}$ is a
$\sigma$\nobreakdash-\hspace{0pt}algebra of its subsets.

By a {\em probability space\/} we understand a triple $\OFP$, where
$(\Omega, \mathfrak{F})$ is a measurable space and $\PP$ is a measure
defined for all $F \in \mathfrak{F}$, with $\PP(\Omega) = 1$ we call
$(\Omega, \mathfrak{F}, \mu)$ a {\em probability space}.

All Banach spaces considered in the paper are real.  $X$ will stand
for a finite\nobreakdash-\hspace{0pt}dimensional Banach space $X$,
with norm $\norm{\cdot}$.

By $X^{*}$ we will denote the dual of $X$, and by $\langle \cdot,
\cdot \rangle$ we denote the standard duality pairing (that is, for
$u \in X$ and $u^{*} \in X^{*}$ the symbol $\langle u, u^{*} \rangle$
denotes the value of the bounded linear functional $u^{*}$ at $u$).
Without further mention, we understand that the norm in $X^{*}$ is
given by $\norm{u^{*}} = \sup\{\,\abs{\langle u, u^{*} \rangle}:
\norm{u} \le 1 \,\}$.

\smallskip
$\TT$ stands for either $\ZZ$ or $\RR$.

For a metric dynamical system $(\OFP,(\theta_t)_{t\in\TT})$, $\Omega'
\subset \Omega$ is {\em invariant\/} if $\theta_{t}(\Omega') =
\Omega'$ for all $t \in \TT$. $(\OFP,(\theta_t)_{t\in\TT})$ is said
to be {\em ergodic\/} if for any invariant $F \in \mathfrak{F}$,
either $\PP(F) = 1$ or $\PP(F) = 0$.

From now on, $(\OFP, \allowbreak (\theta_t)_{t\in\TT})$  (we may
simply write it as $(\theta_{t})_{t \in \TT}$) denotes an ergodic
metric dynamical system.

For $\TT = \RR$ we write $\TT^{+}$ for $[0, \infty)$.  For $\TT =
\ZZ$ we write $\TT^{+}$ for $\{0, 1, 2, 3, \dots\}$. By a {\em
measurable linear skew\nobreakdash-\hspace{0pt}product semidynamical
system\/} \allowbreak $\Phi = \allowbreak ((U_\omega(t))_{\omega \in
\Omega, t \in \TT^{+}}, \allowbreak (\theta_t)_{t\in\TT})$ on  $X$
covering a metric dynamical system $(\theta_{t})_{t \in \TT}$ we
understand a $(\mathfrak{B}(\TT^{+}) \otimes \mathfrak{F} \otimes
\mathfrak{B}(X), \mathfrak{B}(X))$\nobreakdash-\hspace{0pt}measurable
mapping
\begin{equation*}
[\, \TT^{+} \times \Omega \times X \ni (t,\omega,u) \mapsto
U_{\omega}(t)u \in X \,]
\end{equation*} satisfying the following:
\begin{itemize}
\item
\begin{equation}
\label{eq-identity}
U_{\omega}(0) = \mathrm{Id}_{X} \quad \forall \omega \in \Omega,
\end{equation}
\begin{equation}
\label{eq-cocycle}
U_{\theta_{s}\omega}(t) \circ U_{\omega}(s) = U_{\omega}(t+s)
\qquad \forall \omega \in \Omega, \ t, s \in \TT^{+};
\end{equation}
\item
for each $\omega \in \Omega$ and $t \in \TT^{+}$, $[\, X \ni u
\mapsto U_{\omega}(t)u \in X \,] \in \mathcal{L}(X)$.
\end{itemize}
To avoid using lower indices, we usually write $\Phi = (U_\omega(t),
\theta_t)$.

When $\TT^{+} = [0,\infty)$ we call a \mlsps\ a (measurable linear
skew\nobreakdash-\hspace{0pt}product) {\em semiflow\/}.  To emphasize
the situation when $\TT^{+} = \{0, 1, 2, \dots \}$, we speak of
(measurable linear skew\nobreakdash-\hspace{0pt}product) {\em
discrete\nobreakdash-\hspace{0pt}time\/} semidynamical system.

Let $\Phi = ((U_\omega(t))_{\omega \in \Omega, t \in \TT^{+}},
\allowbreak (\theta_t)_{t\in\TT})$ be a measurable linear
skew\nobreakdash-\hspace{0pt}product semidynamical system on $X$
covering $(\theta_{t})_{t \in \TT}$.  For $\omega \in \Omega$, $t \in
\TT^{+}$ and $u^{*} \in X^*$ we define $U^{*}_{\omega}(t)u^{*}$ by
\begin{equation}
\label{dual-definition}
\langle u, U^{*}_{\omega}(t)u^{*} \rangle = \langle
U_{\theta_{-t}\omega}(t)u , u^{*} \rangle \qquad \text{for each } u
\in X
\end{equation}
(in other words, $U^{*}_{\omega}(t)$ is the mapping dual to
$U_{\theta_{-t}\omega}(t)$).  The mapping
\begin{equation*}
[\, \TT^{+} \times \Omega \times X^{*} \ni (t,\omega,u^{*})
\mapsto U^{*}_{\omega}(t)u^{*} \in X^{*} \,]
\end{equation*}
is $(\mathfrak{B}(\TT^{+}) \otimes \mathfrak{F} \otimes
\mathfrak{B}(X^{*}),
\mathfrak{B}(X^{*}))$\nobreakdash-\hspace{0pt}measurable. We call the
\mlsps\ $\Phi^{*} = \allowbreak (U^{*}_\omega(t), \allowbreak
\theta_{-t})$ on $X^{*}$ covering $\theta_{-t}$ the {\em dual\/} of
$\Phi$.

\smallskip
By a {\em cone\/} in $X$ we understand a closed convex set $X^{+}$
such that
\begin{description}
\item{(C1)}
$\alpha \ge 0$ and $u \in X^{+}$ imply ${\alpha}u \in X^{+}$, and
\item{(C2)}
$X^{+} \cap (-X^{+}) = \{0\}$.
\end{description}

A pair $(X, X^{+})$, where $X^{+}$ is a cone in $X$, is referred to
as an {\em ordered Banach space}.

If $(X, X^{+})$ is an ordered Banach space, for $u, v \in X$ we write
$u \le v$ if $v - u \in X^{+}$, and $u < v$ if $u \le v$ and $u \ne
v$.  The symbols $\ge$ and $>$ are used in an analogous way.

Let $X^{+}$ be a cone in $X$. $X^{+}$ is called {\em total\/} if
$\cl(X^{+} - X^{+}) = X$, {\em reproducing\/} if $X^{+} - X^{+} = X$,
and {\em solid \/}if the interior $X^{++}$ of $X^{+}$ is nonempty,
$X^{+}$ is called {\em normal\/} if there exists $C > 0$ such that
for any $u, v \in X$ satisfying $0 \le u \le v$ there holds $\norm{u}
\le C \norm{v}$.  An ordered Banach space $(X,X^{+})$ is called {\em
strongly ordered\/} if $X^{+}$ is solid.

The following lemma is well known.
\begin{lemma}
\label{lemma:Krein-Shmulyan}
Let $X^{+}$ be a  cone in a
finite\nobreakdash-\hspace{0pt}dimensional Banach space $X$ with norm
$\norm{\cdot}$.
\begin{itemize}
\item[{\rm (1)}]
$X^{+}$ is normal.
\item[{\rm (2)}]
$X^{+}$ is solid iff $X^{+}$ is reproducing iff $X^{+}$ is total.
\item[{\rm (3)}]
If $X^{+}$ is reproducing, then there exists $K \ge 1$ with the
property that for each $u \in X$ there are $u^{+}, u^{-} \in
X^{+}$ such that $u = u^{+} - u^{-}$, $\norm{u^{+}} \le K
\norm{u}$, $\norm{u^{-}} \le K \norm{u}$.
\end{itemize}
\end{lemma}

Observe that Lemma \ref{lemma:Krein-Shmulyan}(3) holds for general
Banach spaces (see \cite[Theorem 2.2]{AbAlBu}).

Let $X^{+}$ be a solid cone in a
finite\nobreakdash-\hspace{0pt}dimensional Banach space $X$. By Lemma
\ref{lemma:Krein-Shmulyan}(1), $X^{+}$ is normal.  By appropriately
renorming $X$ we can assume that the norm $\norm{\cdot}$ has the
property that for any $u, v \in X$, $0 \le u \le v$ implies $\norm{u}
\le \norm{v}$ (see \cite[V.3.1, p.~216]{Schaef}).  Such a norm is
called {\em monotonic\/}. From now on, when speaking of a strongly
ordered Banach space we assume that the norm on $X$ is monotonic.

For an ordered Banach space $(X, X^{+})$ denote by $(X^{*})^{+}$ the
set of all $u^{*} \in X^{*}$ such that $\langle u, u^{*} \rangle \ge
0$ for all $u \in X^{+}$.  If $X^{+}$ is solid then $(X^{*})^{+}$ is
a solid cone in $X^{*}$.

\medskip
Assume that $(X, X^{+})$ is an ordered Banach space.  We say that $u,
v \in X^{+} \setminus \{0\}$ are {\em comparable\/}, written $u \sim
v$, if there are positive numbers $\underline{\alpha},
\overline{\alpha}$ such that $\underline{\alpha}v \le u \le
\overline{\alpha}v$.  The $\sim$ relation is clearly an equivalence
relation.  For a nonzero $u \in X^{+}$ we understand by the {\em
component\/} of $u$, denoted by $C_{u}$, the equivalence class of
$u$, $C_{u} = \{\, v \in X^{+} \setminus \{0\}: v \sim u \,\}$.

\begin{example}
\label{finite-dim-example-standard}
{\em Let
\begin{equation*}
X = \{\, u = (u_1, \dots, \allowbreak u_N)^{\top}: u_i \in \RR \text{
for all } 1 \le i \le N \,\}.
\end{equation*}
By the \emph{standard cone\/} in $X$ we understand
\begin{equation*}
X^{+} = \{\, u = (u_1, \dots, u_N)^{\top} \in X: u_i \ge 0 \text{ for
all } 1 \le i \le N \,\}.
\end{equation*}
$X^{+}$ is a solid cone, with interior
\begin{equation*}
X^{++} = \{\, u = (u_1, \dots, u_N)^{\top} \in X: u_i > 0 \text{ for
all } 1 \le i \le N \,\}.
\end{equation*}
Furthermore, $(X,X^{+})$ is a \emph{lattice\/}: any two $u, v \in X$
have a least upper bound $u \vee v$, $(u \vee v)_i = \max\{u_i,
v_i\}$, $1 \le i \le N$, and a greatest lower bound $u \wedge v$, $(u
\wedge v)_i = \min\{u_i, v_i\}$, $1 \le i \le N$.

In the notation of Lemma~\ref{lemma:Krein-Shmulyan}(3) we specify
$u^{+}$ to be equal to $u \wedge 0$, and $u^{-}$ to be equal to $(-u)
\vee 0$.

All $\ell^{p}$\nobreakdash-\hspace{0pt}norms, $1 \le p \le \infty$,
on $X$ are monotonic.  Indeed, if $1 \le p < \infty$ then for any $0
\le u \le v$ one has
\begin{equation*}
\lVert u \rVert_{p}^{p} = \sum\limits_{i = 1}^{N} \abs{u_i}^{p} \le
\sum\limits_{i = 1}^{N} \abs{v_i}^{p} = \lVert v \rVert_{p}^{p}.
\end{equation*}
Moreover, for those norms the constant $K$ in
Lemma~\ref{lemma:Krein-Shmulyan}(3) can be taken to be $1$:
\begin{equation*}
\lVert u^{+} \rVert_{p}^{p} = \sum\limits_{i = 1}^{N} \abs{\max\{u_i,
0\}}^{p} = \sum\limits_{\substack{i \in \{1, \dots, N\} \\ u_i > 0}}
\abs{u_i}^{p} \le \sum\limits_{i = 1}^{N} \abs{u_i}^{p} = \lVert u
\rVert_{p}^{p},
\end{equation*}
and similarly for $u^{-}$. The case $p = \infty$ is considered in an
analogous way.

When we identify the dual space $X^{*}$ with $\RR^{N}$ and the
duality pairing $\langle u, u^{*} \rangle$ with $u^{\top} u^{*}$, the
cone $(X^{*})^{+}$ is given by the same formula as $X^{+}$.}
\end{example}

Now we introduce our assumptions.

\smallskip

\noindent\textbf{(B0)} (Ordered
finite\nobreakdash-\hspace{0pt}dimensional space) {\em $X^{+}$ is a
reproducing cone in a finite\nobreakdash-\hspace{0pt}dimensional
Banach space $X$ \textup{(}this is equivalent to saying that $(X,
X^{+})$ is a strongly ordered
finite\nobreakdash-\hspace{0pt}dimensional Banach space\textup{)},
with $\dim{X} \ge 2$.}

\smallskip
\noindent (Compare (A0)(i)--(iii) in~\cite{MiShPart1}).  We remark
that (A0)(i) and (ii) in \cite{MiShPart1} are automatically satisfied
for a cone $X^{+}$ in finite\nobreakdash-\hspace{0pt}dimensional
Banach space. (A0)(iii) in \cite{MiShPart1} assumes that $(X,X^{+})$
is a Banach lattice, which implies that $X^{+}$ is reproducing or
equivalently solid  and hence (B0) is weaker than (A0)(iii). In
general, the main results in \cite{MiShPart1} still hold if the
assumption that $(X,X^{+})$ is a Banach lattice is replaced by the
assumption that $X^{+}$ is reproducing (see Remarks
\ref{remark-lyp-exp} and \ref{remark-exp-sep}).

\medskip
\noindent \textbf{(B1)} (Integrability/injectivity) {\em $\Phi =
((U_\omega(t))_{\omega \in \Omega, t \in \TT^{+}},
(\theta_t)_{t\in\TT})$ is a \mlsps\ on a
finite\nobreakdash-\hspace{0pt}dimensional Banach space $X$ covering
an ergodic metric dynamical system $(\theta_{t})_{t \in \TT}$ on
$\OFP$, with the complete measure $\PP$ in the case of $\TT = \RR$,
satisfying the following:}
\begin{itemize}
\item[(i)] (Integrability) {\em
\begin{itemize}
\item
In the discrete\nobreakdash-\hspace{0pt}time case: The
function
\begin{equation*}
[\, \Omega \ni \omega \mapsto \lnplus{\norm{U_{\omega}(1)}}
\in [0,\infty) \,]\in L_1(\OFP).
\end{equation*}
\item
In the continuous\nobreakdash-\hspace{0pt}time case: The
functions
\begin{equation*}
[\, \Omega \ni \omega \mapsto \sup\limits_{0 \le s \le 1}
{\lnplus{\norm{U_{\omega}(s)}}} \in [0,\infty) \,]\in
L_1(\OFP)
\end{equation*}
and
\begin{equation*}
[\, \Omega \ni \omega \mapsto \sup\limits_{0 \le s \le 1}
{\lnplus{\norm{U_{\theta_{s}\omega}(1-s)}}} \in [0,\infty)
\,]\in L_1(\OFP).
\end{equation*}
\end{itemize}
}
\item[(ii)] (Injectivity)
{\em For each $\omega \in \Omega$ the linear operator
$U_{\omega}(1)$ is injective.}
\end{itemize}

\noindent (Compare (A1)(i)--(iii) in~\cite{MiShPart1}. Note that
$U_\omega(1)$ is automatically completely continuous in the
finite-dimensional case).

\medskip
\noindent\textbf{(B2)} (Positivity) {\em $(X,X^+)$ is an ordered
finite\nobreakdash-\hspace{0pt}dimensional Banach space and $\Phi =
((U_\omega(t))_{\omega \in \Omega, t \in \TT^{+}}, \allowbreak
(\theta_t)_{t\in\TT})$ is a \mlsps\ on $X$ covering an ergodic metric
dynamical system $(\theta_{t})_{t \in \TT}$ on $\OFP$, satisfying the
following:
\begin{equation*}
U_{\omega}(t)u_1 \le U_{\omega}(t)u_2
\end{equation*}
for any $\omega \in \Omega$, $t \in \TT^{+}$ and $u_1, u_2 \in X$
with $u_1 \le u_2$.}

\noindent (Compare (A2) in~\cite{MiShPart1}).

It follows immediately that if (B2) is satisfied then there holds
\begin{equation*}
U^{*}_{\omega}(t)u_1^{*} \le U^{*}_{\omega}(t)u_2^{*}
\end{equation*}
for any $\omega \in \Omega$, $t \in \TT^{+}$ and $u_1^{*}, u_2^{*}
\in X^{*}$ with $u_1^{*} \le u_2^{*}$.

\medskip
\noindent\textbf{(B3)} (Focusing) {\em \textup{(B2)} is satisfied and
there are $\mathbf{e} \in X^{+}$  with $\norm{\mathbf{e}} = 1$  and
an $(\mathfrak{F},
\mathfrak{B}(\RR))$\nobreakdash-\hspace{0pt}measurable function
$\varkappa \colon \Omega \to [1,\infty)$ with
$\lnplus{\ln{\varkappa}} \in L_1(\OFP)$ such that for any $\omega \in
\Omega$ and any nonzero $u \in X^{+}$ there is $\beta(\omega,u) > 0$
with the property that
\begin{equation*}
\beta(\omega,u) \mathbf{e} \le U_{\omega}(1)u \le \varkappa(\omega)
\beta(\omega,u)\mathbf{e}.
\end{equation*}
}

\noindent (Compare (A3) in~\cite{MiShPart1}).

\smallskip

\noindent\textbf{(B3)$^*$}  (Focusing)  {\em \textup{(B2)} is
satisfied and there are $\mathbf{e}^{*} \in (X^{*})^{+}$ with $
\norm{\mathbf{e}^{*}} = 1$  and an $(\mathfrak{F},
\mathfrak{B}(\RR))$\nobreakdash-\hspace{0pt}measurable function $
\varkappa^{*} \colon \Omega \to [1,\infty)$ with
$\lnplus{\ln{\varkappa^*}} \in L_1(\OFP)$ such that for any $\omega
\in \Omega$ and any nonzero $u^{*} \in (X^{*})^{+}$ there is
$\beta^{*}(\omega,u^{*}) > 0$ with the property that
\begin{equation*}
\beta^{*}(\omega,u^{*}) \mathbf{e}^{*} \le U^{*}_{\omega}(1)u^{*} \le
\varkappa^{*}(\omega) \beta^{*}(\omega,u^{*})\mathbf{e}^{*}.
\end{equation*}
}

\noindent\textbf{(B4)} (Strong focusing) {\em \textup{(B3)},
\textup{(B3)$^*$} are satisfied and  $\ln{\varkappa} \in L_1(\OFP)$,
$\ln{\varkappa^*} \in L_1(\OFP)$, and $\langle \mathbf{e},
\mathbf{e}^* \rangle > 0$.}

\smallskip
\noindent (Compare (A4) in~\cite{MiShPart1}).

\medskip
\noindent\textbf{(B5)} (Strong positivity in one direction) {\em
There are $\mathbf{\overline{e}} \in X^+$ with
$\norm{\mathbf{\overline{e}}} = 1$ and an $(\mathfrak{F},
\mathfrak{B}(\RR))$\nobreakdash-\hspace{0pt}measurable function $\nu
\colon \Omega \to (0,\infty)$, with $\ln^{-}{\nu} \in L_1(\OFP)$,
such that
\begin{equation*}
U_{\omega}(1) \mathbf{\overline{e}} \ge \nu(\omega)
\mathbf{\overline{e}}\quad \forall \ \omega \in \Omega.
\end{equation*}
}

\noindent (Compare (A5) in~\cite{MiShPart1}).

\smallskip

\noindent\textbf{(B5)$^*$} (Strong positivity in one direction) {\em
There are $\mathbf{\bar e^*}\in (X^*)^+$ with
$\norm{\mathbf{\overline{e}^*}} = 1$ and an $(\mathfrak{F},
\mathfrak{B}(\RR))$\nobreakdash-\hspace{0pt}measurable function
$\nu^* \colon \Omega \to (0,\infty)$, with $\ln^{-}{\nu^*} \in
L_1(\OFP)$, such that
\begin{equation*}
U_{\omega}^*(1) \mathbf{\overline{e}^*} \ge \nu^*(\omega)
\mathbf{\overline{e}^*} \quad \forall \ \omega \in \Omega.
\end{equation*}
}

\noindent (Compare (A5)$^*$ in~\cite{MiShPart1}).

\begin{remark}
\label{assumption-rk}
{\em We can replace time $1$ with some nonzero $T$ belonging to
$\TT^{+}$ in (B1), (B3), (B4),  and  (B3)$^*$.}
\end{remark}

We now state the definitions introduced  in Part I. Throughout the
rest of this subsection, we assume $(X,X^{+})$ is a
finite\nobreakdash-\hspace{0pt}dimensional ordered Banach space, and
(B2).

\begin{definition}[Entire orbit]
For $\omega \in \Omega$, by an {\em entire orbit\/} of $U_{\omega}$
we understand a mapping $v_{\omega} \colon \TT \to X$ such that
$v_{\omega}(s + t) = U_{\theta_{s}\omega}(t) v_{\omega}(s)$ for any
$s \in \TT$ and $t \in \TT^{+}$.  The function constantly equal to
zero is referred to as the {\em trivial entire orbit\/}.
\end{definition}
\begin{definition}[Entire positive orbit]
An entire orbit $v_{\omega}$ of $U_{\omega}$ is called an {\em entire
positive orbit\/} if $v_{\omega}(t) \in X^{+}$ for all $t \in \TT$.
An entire positive orbit is {\em nontrivial\/} if $v_{\omega}(t) \in
X^{+} \setminus \{0\}$ for all $t \in \TT$.
\end{definition}
Entire (positive) orbits of $\Phi^*$ are defined in a  similar way.

\medskip
A family $\{E(\omega)\}_{\omega \in \Omega_0}$ of
$l$\nobreakdash-\hspace{0pt}dimensional vector subspaces of $X$ is
{\em measurable\/} if there are $(\mathfrak{F},
\mathfrak{B}(X))$-measurable functions $v_1, \dots, v_l \colon
\Omega_0 \to X$, $\PP(\Omega_0) = 1$, such that $(v_1(\omega), \dots,
v_l(\omega))$ forms a basis of $E(\omega)$ for each $\omega \in
\Omega_0$.

Let $\{E(\omega)\}_{\omega \in \Omega_0}$ be a family of
$l$\nobreakdash-\hspace{0pt}dimensional vector subspaces of $X$, and
let $\{F(\omega)\}_{\omega \in \Omega_0}$ be a family of
$l$\nobreakdash-\hspace{0pt}codimensional vector subspaces of $X$,
such that $E(\omega) \oplus F(\omega) = X$ for all $\omega \in
\Omega_0$, $\PP(\Omega_0) = 1$.  We define the {\em family of
projections associated with the decomposition\/} $E(\omega) \oplus
F(\omega) = X$ as $\{P(\omega)\}_{\omega \in \Omega_0}$, where
$P(\omega)$ is the linear projection of $X$ onto $F(\omega)$ along
$E(\omega)$, for each $\omega \in \Omega_0$.

The following remark is in~order regarding measurability of
decomposition.  In Part~I (\cite{MiShPart1}) of the series, as well
as in, for example, \cite{Lian-Lu}, when a decomposition $E(\omega)
\oplus F(\omega) = X$ of a (perhaps infinite-dimensional) Banach
space $X$ is considered, where $E(\omega)$ has finite dimension $l$,
the assumption is that the family $\{E(\omega)\}_{\omega \in
\Omega_0}$ is measurable and that the family of projections
$\{P(\omega)\}_{\omega \in \Omega_0}$ is {\em strongly measurable\/}:
for each $u \in X$ the mapping $[\, \Omega_0 \ni \omega \mapsto
P(\omega)u \in X \,]$ is $(\mathfrak{F},
\mathfrak{B}(X))$\nobreakdash-\hspace{0pt}measurable.  But in the
finite-dimensional case this is equivalent to saying that both
families $\{E(\omega)\}_{\omega \in \Omega_0}$ and
$\{F(\omega)\}_{\omega \in \Omega_0}$ are measurable.  We call such a
decomposition a {\em measurable decomposition\/}.

We say that the decomposition $E(\omega) \oplus F(\omega) = X$ is
{\em invariant\/} if $\Omega_0$ is invariant, $U_{\omega}(t)E(\omega)
= E(\theta_{t}\omega)$ and $U_{\omega}(t)F(\omega) \subset
F(\theta_{t}\omega)$, for each $t \in \TT^{+}$.

A (strongly measurable) family of projections associated with the
invariant measurable decomposition $E(\omega) \oplus F(\omega) = X$
is referred to as {\em tempered\/} if
\begin{equation*}
\lim\limits_{\substack{t \to \pm\infty \\ t \in \TT}}
\frac{\ln{\norm{P(\theta_{t}\omega)}}}{t} = 0 \qquad \PP\text{-a.s.
on }\Omega_0.
\end{equation*}

\begin{definition}[Generalized principal Floquet subspaces and
principal Lyapunov exponent]
\label{generalized-floquet-space}
A family of one\nobreakdash-\hspace{0pt}dimensional subspaces
$\{\tilde{E}(\omega)\}_{\omega \in \tilde{\Omega}}$ of $X$ is called
a family of {\em generalized principal Floquet subspaces} of $\Phi =
(U_\omega(t), \theta_t)$ if $\tilde{\Omega} \subset \Omega$ is
invariant, $\PP(\tilde{\Omega}) = 1$, and
\begin{itemize}
\item[{\rm (i)}]
$\tilde{E}(\omega) = \spanned{\{w(\omega)\}}$ with $w \colon
\tilde{\Omega} \to X^+ \setminus \{0\}$ being $(\mathfrak{F},
\mathfrak{B}(X))$\nobreakdash-\hspace{0pt}measurable,
\item[{\rm (ii)}]
$U_{\omega}(t) \tilde{E}(\omega) = \tilde{E}(\theta_{t}\omega)$,
for any $\omega \in \tilde{\Omega}$ and any $t \in \TT^{+}$,
\item[{\rm (iii)}]
there is $\tilde{\lambda} \in [-\infty, \infty)$ such that
\begin{equation*}
\tilde{\lambda} = \lim_{\substack{t\to\infty \\ t \in \TT^{+}}}
\frac{1}{t} \ln{\norm{U_\omega(t)w(\omega)}}\quad \forall
\omega\in  \tilde{\Omega},
\end{equation*}
 and
\item[{\rm (iv)}]
\begin{equation*}
\limsup_{\substack{t\to\infty \\ t \in \TT^{+}}} \frac{1}{t}
\ln{\norm{U_\omega(t)u}} \le \tilde{\lambda}\quad \forall \omega
\in \tilde{\Omega}\,\,\,\text{and}\,\, \, \forall u \in X
\setminus \{0\}.
\end{equation*}
\end{itemize}
$\tilde{\lambda}$ is called the {\em generalized principal Lyapunov
exponent} of $\Phi$ associated to the generalized principal Floquet
subspaces $\{\tilde E(\omega)\}_{\omega\in\tilde\Omega}$.
\end{definition}

Observe that if $\{\tilde E(\omega)\}_{\omega\tilde\in\Omega}$ is a
family of generalized principal Floquet subspaces of $ (U_\omega(t),
\theta_t)$, then for any $\omega \in \tilde{\Omega}$ the function
$v_\omega \colon \TT \to X^{+} \setminus \{0\}$,
\begin{equation}
\label{eq:entire-positive}
v_\omega(t) = \begin{cases}
U_\omega(t) w(\omega), & \quad t\in\TT, \ t \ge 0 \\[0.5ex]
\displaystyle
\frac{\norm{w(\omega)}}{\norm{U_{\theta_t\omega}(-t)w(\theta_t\omega)}}
w(\theta_t\omega), & \quad t \in \TT, \ t < 0
\end{cases}
\end{equation}
is a nontrivial entire positive orbit of $U_{\omega}$.

\begin{definition}[Generalized exponential separation]
\label{generalized-exponential-separation}
$\Phi = (U_\omega(t), \theta_t)$ is said to admit a {\em generalized
exponential separation\/} if there are a family of generalized
principal Floquet subspaces $\{\tilde{E}(\omega)\}_{\omega \in
\tilde{\Omega}}$ and a measurable family of
one\nobreakdash-\hspace{0pt}codimensional subspaces
$\{\tilde{F}(\omega)\}_{\omega \in \tilde{\Omega}}$ of $X$ satisfying
the following
\begin{itemize}
\item[{\rm (i)}]
$\tilde F(\omega) \cap X^{+} = \{0\}$ for any $\omega \in
\tilde{\Omega}$,
\item[{\rm (ii)}]
$X = \tilde{E}(\omega) \oplus \tilde{F}(\omega)$ for any
$\omega\in\tilde\Omega$, where the decomposition is invariant,
and the family of projections associated with this decomposition
is tempered,
\item[{\rm (iii)}]
there exists $\tilde{\sigma} \in (0, \infty]$ such that
\begin{equation*}
\lim_{\substack{t\to\infty \\ t \in \TT^{+}}} \frac{1}{t}
\ln{\frac{\norm{U_{\omega}(t)|_{\tilde{F}(\omega)}}}
{\norm{U_{\omega}(t) w(\omega)}}} = -\tilde{\sigma}\quad \forall
\omega \in \tilde{\Omega}.
\end{equation*}
\end{itemize}
We say that $\{\tilde{E}(\cdot), \tilde{F}(\cdot), \tilde{\sigma} \}$
{\em generates a generalized exponential separation}.
\end{definition}

\medskip
We end this subsection with the following theorem, which follows from
the Oseledets-type theorems proved in \cite{Rue1}.

\begin{theorem}
\label{Oseledets-thm}
Let $X$ be a finite\nobreakdash-\hspace{0pt}dimensional Banach space,
and let $\Phi$ be a \mlsps\ satisfying \textup{(B1)(i)}. Then there
exists an invariant $\Omega_0 \subset \Omega$, $\PP(\Omega_0) = 1$,
and $\lambda_1 \in [-\infty,\infty)$ with the property that
\begin{equation}
\label{lambda-1-eq-0}
\lim\limits_{\substack{t\to\infty \\ t \in \TT^{+}}}
\frac{\ln{\norm{U_{\omega}(t)}}}{t} = \lambda_1\quad \forall \omega
\in \Omega_0.
\end{equation}
Moreover, either
\begin{itemize}
\item[{\rm (i)}]
\begin{equation}
\label{lambda-1-eq-1}
\lim\limits_{\substack{t\to\infty \\ t \in \TT^{+}}}
\frac{\ln{\norm{U_{\omega}(t)u}}}{t} = \lambda_1 \quad \forall
\omega \in \Omega_0\,\,\,\text{and}\,\,\, \forall u \in X
\setminus \{0\},
\end{equation}
\end{itemize}
or
\begin{itemize}
\item[{\rm (ii)}]
there exist $\hat{\lambda}_2 < \lambda_1$, and a measurable
family of linear subspaces $\{\hat{F}_1(\omega)\}_{\omega \in
\Omega_0}$ of $X$ such that
\begin{itemize}[\textbullet]
\item
$U_{\omega}(t) \hat{F}_1(\omega) \subset
\hat{F}_1(\theta_{t}\omega)$, for all $\omega \in \Omega_0$
and all $t \in \TT^{+}$,
\item
\begin{equation}
\label{lambda-1-eq-2}
\lim\limits_{\substack{t\to\infty \\ t \in \TT^{+}}}
\frac{\ln{\norm{U_{\omega}(t)u}}}{t} = \lambda_1 \quad \forall
\omega \in \Omega_0\,\,\,\text{and}\,\,\, \forall u \in X
\setminus \hat{F}_1(\omega),
\end{equation}
\item
\begin{equation}
\label{lambda-2-eq}
\lim\limits_{\substack{t\to\infty \\ t \in \TT^{+}}}
\frac{\ln{\norm{U_{\omega}(t)|_{\hat F_1(\omega)}}}}{t} =
\hat{\lambda}_2\quad \forall \omega \in \Omega_0.
\end{equation}
\end{itemize}
Moreover, if \textup{(B1)(ii)} holds, then there is a measurable
family of linear subspaces $\{E_1(\omega)\}_{\omega \in
\Omega_0}$ of $X$ such that
\begin{itemize}[\textbullet]
\item
$U_{\omega}(t) E_1(\omega) = E_1(\theta_{t}\omega)$, for all
$\omega \in \Omega_0$ and all $t \in \TT^{+}$,
\item
$X = E_1(\omega) \oplus \hat{F}_1(\omega)$ for any $\omega
\in \Omega_0$, where the family of projections associated
with this invariant decomposition is tempered,
\item
\begin{equation}
\label{lambda-1-eq-3}
\lim\limits_{\substack{t\to\pm\infty \\ t \in \TT}}
\frac{\ln{\norm{U_{\omega}(t)u}}}{t} = \lambda_1 \quad
\forall \omega \in \Omega_0 \text{ and } \forall u \in
E_1(\omega) \setminus \{0\}.
\end{equation}
\end{itemize}
\end{itemize}
\end{theorem}

\subsection{General Theorems}
\label{subsection:general-theorems}

In this subsection, we state some general theorems, most of which are
established in Part I (\cite{MiShPart1}).

Throughout this subsection, we assume that $(X,X^{+})$ is an ordered
finite\nobreakdash-\hspace{0pt}dimensional Banach space. The first
theorem concerns the existence of entire positive solutions, which
partially follows from \cite[Theorem 3.5]{MiShPart1}.

\begin{theorem}[Entire positive orbits]
\label{thm:entire-orbits-existence}
Assume  \textup{(B1)(i)} and \textup{(B2)}.
\begin{enumerate}
\item[{\rm (i)}]
Let $U_{\omega}(t) (X^{+} \setminus \{0\}) \subset X^{+}
\setminus \{0\}$ for all $\omega \in \Omega$ and all $t \in
\TT^{+}$. Then for each $\omega \in \Omega$ there exists an
entire positive orbit $v_{\omega} \colon \TT \to X^{+} \setminus
\{0\}$ of $U_\omega$.
\item[{\rm (ii)}]
Assume moreover that $X^{+}$ is reproducing.  Let
\textup{(B1)(ii)} be satisfied.  If (ii) in
Theorem~\ref{Oseledets-thm} holds then for each $\omega \in
\Omega_0$ an entire positive orbit can be chosen so that
$v_{\omega}(t) \in X^{+} \cap E_1(\theta_{t}\omega)$ for all $t
\in \TT$.
\end{enumerate}
\end{theorem}
\begin{proof}
(i)  Denote $\mathcal{S}_1(X^{+}) := \{\, u \in X^{+}: \norm{u} = 1
\,\}$.

Fix $\omega \in \Omega$.  For $(m,n) \in \ZZ^2$ such that $0 \le n
\le m$ define the mapping $\mathcal{U}(m;n) \colon
\mathcal{S}_1(X^{+}) \to \mathcal{S}_1(X^{+})$ by the formula
\begin{equation*}
\mathcal{U}(m;n)(u) := \frac{U_{\theta_{-m}\omega}(m-n) u}
{\norm{U_{\theta_{-m}\omega}(m-n) u}}.
\end{equation*}
By the assumptions, the mappings $\mathcal{U}(m;n)$ are well defined
and continuous.  It follows from~\eqref{eq-cocycle} that
$\mathcal{U}(m_2;n) = \mathcal{U}(m_1;n) \circ \mathcal{U}(m_2;m_1)$,
consequently $\Image(\mathcal{U}(m_2;n)) \allowbreak \subset
\Image(\mathcal{U}(m_1;n))$, for any $0 \le n \le m_1 \le m_2$.

For $n = 0, 1, 2, \dots$ let
\begin{equation*}
G_n := \bigcap\limits_{m=n}^{\infty} \Image(\mathcal{U}(m;n)).
\end{equation*}
$G_n$, as the intersection of a nonincreasing family of compact
nonempty sets, is compact and nonempty, too.

Notice that $u \in \mathcal{S}_1(X^{+})$ belongs to $G_n$ if~and
only~if there is a sequence $(u^{(m)})_{m=n}^{\infty}$ with $u^{(n)}
= u$ such that $u^{(m)} \in \mathcal{S}_1(X^{+})$ and $u^{(m)} =
\mathcal{U}(m+1;m)u^{(m+1)}$ for each $m = n, n+1, n+2, \dots$.

It suffices now to pick one $u \in G_0$ and
$\{u^{(m)}\}_{m=0}^\infty\subset \mathcal{S}_1(X^{+})$ with $u^{(0)}
= u$ and $u^{(m)} = \mathcal{U}(m+1;m)u^{(m+1)}$ for each $m = 0,1,2,
\dots$. Then put
\begin{equation*}
v_{\omega}(n) :=
\begin{cases}
U_{\omega}(n) u^{(0)} & \quad \text{for }  n = 0, 1, 2, \dots \\[1ex]
\displaystyle \frac{u^{(-n)}}{\norm{U_{\theta_{-1}\omega}(1) u^{(1)}}
\norm{U_{\theta_{-2}\omega}(1) u^{(2)}} \dots
\norm{U_{\theta_{n}\omega}(1) u^{(-n)}}} & \quad \text{for } n = -1,
-2, \dots
\end{cases}
\end{equation*}
This concludes the proof of~(i) in the
discrete\nobreakdash-\hspace{0pt}time case.  In the
continuous\nobreakdash-\hspace{0pt}time case one puts $v_{\omega}(t)
:= U_{\theta_{\intpart{t}}\omega}(t - \intpart{t})
v_{\omega}(\intpart{t})$ for any $\omega \in \Omega$ and any $t \in
\RR \setminus \ZZ$.

(ii) follows from~\cite[Theorem 3.5]{MiShPart1}.
\end{proof}

The next theorem shows the existence of generalized Floquet subspaces
and principal Lyapunov exponent, which partially follows from
\cite[Theorem 3.6]{MiShPart1}.
\begin{theorem}[Generalized principal Floquet subspace and Lyapunov exponent]
\label{theorem-w}
Assume \textup{(B1)(i)}, \textup{(B2)} and \textup{(B3)}.  Then there
exist an invariant set $\tilde{\Omega}_1 \subset \Omega$,
$\PP(\tilde{\Omega}_1) = 1$, and an $(\mathfrak{F},
\mathfrak{B}(X))$\nobreakdash-\hspace{0pt}measurable function $w
\colon \tilde{\Omega}_1 \to X$, $w(\omega) \in C_{\mathbf{e}}$ and
$\norm{w(\omega)} = 1$ for all $\omega \in \tilde{\Omega}_1$, having
the following properties:
\begin{itemize}
\item[{\rm (1)}]
\begin{equation*}
w(\theta_{t}\omega) =
\frac{U_{\omega}(t)w(\omega)}{\norm{U_{\omega}(t)w(\omega)}}
\end{equation*}
for any $\omega \in \tilde{\Omega}_1$ and $t \in \TT^{+}$.

\item[{\rm (2)}]
Let for some $\omega \in \tilde{\Omega}_1$ a function $v_{\omega}
\colon \TT \to X^{+} \setminus \{0\}$ be an entire orbit of
$U_{\omega}$.  Then $v_{\omega}(t) = \norm{v_{\omega}(0)}
w_{\omega}(t)$ for all $t \in \TT$, where
\begin{equation*}
w_{\omega}(t) := \begin{cases}
(U_{\theta_{t}\omega}(-t)|_{\tilde{E}_1(\theta_{t}\omega)})^{-1}
w(\omega) & \qquad \text{for } t \in \TT, \ t < 0 \\
U_{\omega}(t) w(\omega) & \qquad \text{for } t \in \TT^{+},
\end{cases}
\end{equation*}
with $\tilde E_1(\omega) = \spanned\{w(\omega)\}$.
\item[{\rm (3)}]
There exists $\tilde{\lambda}_1 \in [-\infty,\infty)$ such that
\begin{equation*}
\tilde{\lambda}_1 = \lim_{\substack{t\to \pm\infty \\ t \in
\TT}} \frac{1}{t} \ln{\rho_{t}(\omega)} =
\int\limits_{\Omega} \ln{\rho_{1}}(\omega)\, d\PP
\end{equation*}
for each $\omega \in \tilde{\Omega}_1$, where
\begin{equation*}
\rho_t(\omega) := \begin{cases} \norm{U_\omega(t)w(\omega)} & \quad
\text{for } t \ge 0, \\
1/\norm{U_{\theta_{t}\omega}(-t)w(\theta_{t}\omega)} & \quad
\text{for } t < 0.
\end{cases}
\end{equation*}
\item[{\rm (4)}] Assume, moreover, \textup{(B0)}. Then
$\tilde{\lambda}_1 = \lambda_1$, where $\lambda_1$ is as
in~Theorem~\ref{Oseledets-thm}.  In~particular, for any $u \in X
\setminus \{0\}$,
\begin{equation*}
\limsup_{\substack{t \to \infty \\ t \in \TT^{+}}} \frac{1}{t}
\ln{\norm{U_\omega(t)u}} \le \tilde{\lambda}_1,
\end{equation*}
and then $\{\tilde E_1(\omega)\}_{\omega \in \tilde{\Omega}_1}$
is a family of generalized Floquet subspaces.
\end{itemize}
\end{theorem}
\begin{proof}
Parts (1), (2) and (3) just correspond to parts (1), (2) and (3)
of~\cite[Theorem 3.6]{MiShPart1}.

(4) By~\cite[Proposition~5.5(2)]{MiShPart1}, there exists
$\tilde{\sigma}_1 > 0$ such that
\begin{equation*}
\limsup_{\substack{t \to \infty \\ t \in \TT^{+}}} \frac{1}{t}
\ln{\left\lVert \frac{U_{\omega}(t)u}{\norm{U_{\omega}(t)u}} -
w(\theta_{t}\omega) \right\rVert} \le -\tilde{\sigma}_1
\end{equation*}
for any nonzero $u \in X^{+}$, $\PP$\nobreakdash-\hspace{0pt}a.e.\ on
$\Omega$, which, combined with (3), gives that
\begin{equation*}
\limsup_{\substack{t \to \infty \\ t \in \TT^{+}}} \frac{1}{t}
\ln{\norm{U_\omega(t)u}} \le \tilde{\lambda}_1
\end{equation*}
for any nonzero $u \in X^{+}$.  Since $X = X^{+} - X^{+}$, the above
inequality is satisfied for any nonzero $u \in X$.  It suffices now
to apply Theorem~\ref{Oseledets-thm}.
\end{proof}

\begin{remark}
\label{remark-lyp-exp}
{\em Theorem \ref{theorem-w}(4) is apparently a stronger version of
\cite[Theorem~3.6(5)]{MiShPart1}: In \cite[Theorem 3.6(5)]{MiShPart1}
$(X,X^{+})$ is assumed to be a Banach lattice, which implies that
$X^{+}$ is reproducing. We point out that \cite[Theorem
3.6(5)]{MiShPart1} in fact also holds if the assumption that
$(X,X^{+})$ is a Banach lattice is replaced by the assumption that
$X^{+}$ is reproducing.}
\end{remark}

If  \textup{(B1)(i)}, \textup{(B2)} and \textup{(B3)$^{*}$} are
satisfied, the counterpart of Theorem~\ref{theorem-w} for the dual
$\Phi^{*}$ states, among~others, the existence of an invariant
$\tilde{\Omega}^{*}_1 \subset \Omega$, $\PP(\tilde{\Omega}^{*}_1) =
1$, and an $(\mathfrak{F},
\mathfrak{B}(X^{*}))$\nobreakdash-\hspace{0pt}measurable function
$w^{*} \colon \tilde{\Omega}^{*}_1 \to X^{*}$, $w^{*}(\omega) \in
C_{\mathbf{e}^{*}}$ and $\norm{w^{*}(\omega)} = 1$ for all $\omega
\in \tilde{\Omega}^{*}_1$, satisfying $w^*(\theta_{-t}\omega) =
U_\omega^*(t)w^*(\omega)/\norm{U_\omega^*(t)w^*(\omega)}$ for all
$\omega \in \tilde{\Omega}^{*}_1$ and all $t \in \TT^{+}$.  For
$\omega \in \tilde{\Omega}^{*}_1$, define $\tilde{F}_1(\omega) :=
\{\, u \in X: \langle u, w^{*}(\omega) \rangle = 0 \,\}$.  Then
$\{\tilde{F}_1(\omega)\}_{\omega \in \tilde{\Omega}^{*}_1}$ is a
family of one\nobreakdash-\hspace{0pt}codimensional subspaces of $X$,
such that $U_{\omega}(t) \tilde{F}_1(\omega) \subset
\tilde{F}_1(\theta_{t}\omega)$ for any $\omega \in
\tilde{\Omega}^{*}_1$ and any $t \in \TT^{+}$.

The last theorem is about the existence of generalized exponential
separation, which partially follows from \cite[Theorem
3.8]{MiShPart1}.

\begin{theorem}[Generalized exponential separation]
\label{separation-thm}
Assume \textup{(B0)}, \textup{(B1)(i)}, \textup{(B2)}, and
\textup{(B4)}.  Then there is an invariant set $\tilde{\Omega}_0$,
$\PP(\tilde{\Omega}_0) = 1$, having the following properties.
\begin{itemize}
\item[{\rm (1)}]
The family $\{\tilde{P}(\omega)\}_{\omega \in \tilde{\Omega}_0}$
of projections associated with the measurable invariant
decomposition $\tilde{E}_1(\omega) \oplus \tilde{F}_1(\omega) =
X$ is tempered.

\item[{\rm (2)}]
$\tilde{F_1}(\omega) \cap X^{+} = \{0\}$ for any $\omega \in
\tilde{\Omega}_0$.
\item[{\rm (3)}]
For any $\omega \in \tilde{\Omega}_0$ and any $u \in X \setminus
\tilde{F}_1(\omega)$ \textup{(}in~particular, for any nonzero $u
\in X^{+}$\textup{)} there holds
\begin{equation*}
\lim_{\substack{t\to\infty \\ t \in \TT^{+}}} \frac{1}{t}
\ln{\norm{U_{\omega}(t)}} = \lim_{\substack{t\to\infty \\ t \in
\TT^{+}}} \frac{1}{t} \ln{\norm{U_{\omega}(t)u}} =
\tilde{\lambda}_1.
\end{equation*}
\item[{\rm (4)}]
There exist $\tilde{\sigma} \in (0, \infty]$ and
$\tilde{\lambda}_2 \in [-\infty,\infty)$, $\tilde{\lambda}_2 =
\tilde{\lambda}_1 - \tilde{\sigma}$, such that
\begin{equation*}
\lim_{\substack{t\to\infty \\ t \in \TT^{+}}} \frac{1}{t}
\ln{\frac{\norm{U_{\omega}(t)|_{\tilde{F}_1(\omega)}}}
{\norm{U_{\omega}(t) w(\omega)}}} = - \tilde{\sigma}
\end{equation*}
and
\begin{equation*}
\lim_{\substack{t\to\infty \\ t \in \TT^{+}}} \frac{1}{t}
\ln{\norm{U_{\omega}(t)|_{\tilde{F}_1(\omega)}}} =
\tilde{\lambda}_2
\end{equation*}
for each $\omega \in \tilde{\Omega}_0$.  Hence $\Phi$ admits a
generalized exponential separation.
\item[{\rm (5)}]
Assume that (ii) in Theorem~\ref{Oseledets-thm} holds.  Then
$\tilde{\lambda}_2 = \hat{\lambda}_2 < \tilde{\lambda}_1$ and
$\hat{F}_1(\omega) = \tilde{F}_1(\omega)$ for
$\PP$\nobreakdash-\hspace{0pt}a.e.\ $\omega \in \Omega_{0}$. If,
moreover, \textup{(B1)(ii)} holds then $E_1(\omega) =
\tilde{E}_1(\omega)$ for $\PP$\nobreakdash-\hspace{0pt}a.e.\
$\omega \in \Omega_{0}$.

\item[{\rm (6)}]
If \textup{(B5)} or \textup{(B5)$^*$} holds, then
$\tilde{\lambda}_1 > -\infty$.
\end{itemize}
\end{theorem}
\begin{proof}
The proofs of parts (1) through (4) and of (6) go along the lines of
proofs of the corresponding parts of~\cite[Theorem~3.8]{MiShPart1},
the only difference being that in the present paper we do not assume
that $(X, X^{+})$ is a Banach lattice.  However, by
Lemma~\ref{lemma:Krein-Shmulyan}(3), there exists $K > 0$ such that
any $u \in X$ with $\norm{u} = 1$ can be written as $u^{+} - u^{-}$,
with $u^{+}, u^{-} \in X^{+}$, $\norm{u^{+}} \le K$, $\norm{u^{-}}
\le K$ (cf.\ the proof of~\cite[Proposition~5.11]{MiShPart1}, where
this property is needed).

Regarding the first part of~(5), observe that nowhere in the proof
in~\cite{MiShPart1} of the fact that $\hat{F}_1(\omega) =
\tilde{F}_1(\omega)$ $\PP$\nobreakdash-\hspace{0pt}a.e.\ the
injectivity ((B1)(ii)) is needed.
\end{proof}

\begin{remark}
\label{remark-exp-sep}
{\em Theorem \ref{separation-thm} is apparently a stronger version of
\cite[Theorem~3.8]{MiShPart1}: It is assumed in \cite[Theorem
3.8]{MiShPart1} that $(X,X^{+})$ is a Banach lattice, which implies
that $X^{+}$ is reproducing.  We point out that \cite[Theorem
3.8]{MiShPart1} also holds if the assumption that $(X,X^{+})$ is a
Banach lattice is replaced by the assumption that $X^{+}$ is
reproducing.}
\end{remark}

\section{Positive Matrix Random Dynamical Systems}
\label{section-matrix}
In the present section we consider applications of the general
results stated in Section~\ref{general-theory} to
discrete\nobreakdash-\hspace{0pt}time
finite\nobreakdash-\hspace{0pt}dimensional random dynamical systems
arising from positive random matrices.

Throughout this section $X$ stands for the
$N$\nobreakdash-\hspace{0pt}dimensional real vector space of column
vectors, with $N \ge 2$.

We assume that the dual space $X^{*}$ consists of column vectors,
too, and that the duality pairing between $X$ and $X^{*}$ is given by
the standard inner product (denoted also by the symbol $\langle
\cdot, \cdot \rangle$).  Consequently, $\langle u, u^{*} \rangle =
u^{\top} u^{*}$, for any $u \in X$, $u^{*} \in X^{*}$, where
$^{\top}$ denotes the matrix transpose.

$X^{+}$ is the standard {\em nonnegative cone\/}, $X^{+} = \{\, u =
(u_1, \dots, u_{N})^{\top}: u_i \ge 0$ for all $1 \le i \le N \,\}$,
and $X^{++}$ is its interior, $X^{++} = \{\, u = (u_1, \dots,
u_{N})^{\top}: u_i > 0$ for all $1 \le i \le N \,\}$.

Denote by $(\mathbf{e}_1, \dots, \mathbf{e}_N)$ the standard basis of
$X$.

\smallskip
We will identify a discrete\nobreakdash-\hspace{0pt}time \mlsps\
$\Phi = \allowbreak ((U_{\omega}(n))_{\omega \in \Omega, n \in
\ZZ^{+}}, \allowbreak (\theta_n)_{n \in \ZZ})$ on $X$ covering a
metric dynamical system $(\theta_{n})_{n \in \ZZ}$ with a $(\PP,
\mathfrak{B}(\RR^{N \times N}))$-measurable family
$(S(\omega))_{\omega \in \Omega}$ of $N \times N$ matrices
\begin{equation}
\label{finite-dim-matrix}
S(\omega) =
\left(\begin{matrix}
s_{11}(\omega)&s_{12}(\omega)&\cdots&s_{1N}(\omega) \\
s_{21}(\omega)&s_{22}(\omega)&\cdots&s_{2N}(\omega) \\
\vdots & \vdots & \ddots& \vdots \\
s_{N1}(\omega)&s_{N2}(\omega)&\cdots&s_{NN}(\omega)
\end{matrix}\right),
\end{equation}
that is, $U_{\omega}(1) = S(\omega)$.  Also, we will write $\theta$
instead~of $\theta_{1}$.

We will use the notation
\begin{equation*}
S^{(n)}(\omega) := S(\theta^{n-1}\omega) S(\theta^{n-2}\omega) \cdots
S(\omega), \quad \omega \in \Omega, \ n \ge 1
\end{equation*}
(obviously $S^{(0)}(\omega)$ equals the identity matrix).
Eq.~\eqref{eq-cocycle} takes now the form
\begin{equation}
\label{eq-cocycle-matrix}
S^{(m+n)}(\omega) = S^{(n)}(\theta^{m}\omega) S^{(m)}(\omega), \quad
\omega \in \Omega, \ m, n \in \ZZ^{+}.
\end{equation}
For the dual system $(S^{*}(\omega))_{\omega \in \Omega}$ it follows
from~\eqref{dual-definition} that $S^{*}(\omega) =
(S(\theta^{-1}\omega))^{\top}$.

\medskip
Let
\begin{equation*}
m_{c,i}(\omega) := \min_{1 \leq j\leq N}s_{ji}(\omega), \quad
M_{c,i}(\omega) = \max_{1 \le j \leq N} s_{ji}(\omega),
\end{equation*}
and
\begin{equation*}
m_{r,i}(\omega) := \min_{1\leq j\leq N} s_{ij}(\omega), \quad
M_{r,i}(\omega) := \max_{1\leq j\leq N} s_{ij}(\omega)
\end{equation*}
for $\omega \in \Omega$ and $i = 1, 2, \dots, N$.  Further, let
\begin{equation*}
m_r(\omega) := \min_{1\leq i\leq N} \sum_{j=1}^N s_{ij}(\omega),
\quad m_c(\omega) := \min_{1\leq j\leq N} \sum_{i=1}^N
s_{ij}(\omega).
\end{equation*}
Finally, let
\begin{equation*}
m(\omega) := \min_{1 \le i,j \le N} s_{ij}(\omega), \quad M(\omega)
:= \max_{1 \le i,j \le N} s_{ij}(\omega).
\end{equation*}

We state the following standing assumptions.

\medskip
\noindent \textbf{(D1)} (Positivity, injectivity and integrability)
{\em $(\theta^{n})_{n \in \ZZ}$ is an ergodic metric dynamical system
on $\OFP$, and $S = (s_{ij})_{i,j=1}^{N} \colon \Omega \to \RR^{N
\times N}$ is an $(\mathfrak{F}, \mathfrak{B}(\RR^{N \times
N}))$-measurable matrix function such that
\begin{itemize}
\item[{\rm (i)}]
$s_{ij}(\omega) \ge 0$ for all $\omega \in \Omega$ and $i, j =
1,2,\dots,N$.

\item[{\rm (ii)}]
For each $\omega \in \Omega$, $S(\omega)$ is a
non\nobreakdash-\hspace{0pt}singular matrix.

\item[{\rm (iii)}]
(i) holds and $\lnplus{M} \in L_1(\OFP)$.
\end{itemize}
}

\noindent \textbf{(D2)} (Focusing) {\em
\begin{itemize}
\item[{\rm (i)}]
$s_{ij}(\omega) > 0$ for all $\omega \in \Omega$ and $i,j =
1,2,\dots,N$ and $\lnplus(\ln M_{c,i} - \ln m_{c,i}) \in
L_1(\OFP)$ for $i = 1,2,\dots,N$.

\item[{\rm (ii)}]
$s_{ij}(\omega) > 0$ for all $\omega \in \Omega$ and $i, j =
1,2,\dots,N$ and $\lnplus(\ln M_{r,i} - \ln m_{r,i}) \in
L_1(\OFP)$ for $i = 1,2,\dots,N$.

\item[{\rm (iii)}]
$s_{ij}(\omega) > 0$ for all $\omega \in \Omega$ and $i, j =
1,2,\dots,N$ and $\ln M_{c,i} - \ln m_{c,i} \in L_1(\OFP)$, $\ln
M_{r,i} - \ln m_{r,i} \in L_1(\OFP)$ for $i = 1,2,\dots,N$.
\end{itemize}
}

\noindent \textbf{(D3)} (Strong positivity in one direction) {\em
\begin{itemize}
\item[{\rm (i)}]
$m_r(\omega) > 0$ for each $\omega \in \Omega$ and $\ln^{-}{m_r}
\in L_1(\OFP)$

\item[{\rm (ii)}]
$m_c(\omega) > 0$ for each $\omega \in \Omega$ and $\ln^{-}{m_c}
\in L_1(\OFP)$.
\end{itemize}}

\begin{proposition}[Positivity, injectivity and
integrability]
\label{matrix-prop1}
\begin{itemize}
\item[{\rm (1)}]
\textup{(B2)} holds if and only if \textup{(D1)(i)} holds.
\item[{\rm (2)}]
\textup{(B1)(ii)} holds if and only if \textup{(D1)(ii)} holds.
\item[{\rm (3)}]
Assume \textup{(D1)(i)}.  Then \textup{(B1)(i)} holds if and only
if \textup{(D1)(iii)} holds.
\end{itemize}
\end{proposition}
\begin{proof}
(1) and (2) are obvious.

(3) Since $M(\omega) \le \norm{S(\omega)} \le \sqrt{N} M(\omega)$,
the satisfaction of (B1)(i) is equivalent to $\lnplus{M} \in
L_1(\OFP)$.
\end{proof}

\begin{proposition}[Focusing]
\label{matrix-prop2}
Assume that \textup{(D1)(i)} holds.
\begin{itemize}
\item[{\rm (1)}]
\textup{(B3)} holds with $\mathbf{e} \in X^{++}$ if and only if
\textup{(D2)(i)} holds.
\item[{\rm (2)}]
\textup{(B3)$^*$} holds with $\mathbf{e}^{*} \in X^{++}$ if and
only if  \textup{(D2)(ii)} holds.
\item[{\rm (3)}]
\textup{(B4)} holds with $\mathbf{e}, \mathbf{e}^{*} \in X^{++}$
if and only if \textup{(D2)(iii)} holds.
\end{itemize}
\end{proposition}
\begin{proof}
(1) Suppose that (B3) is satisfied with some $\mathbf{e} \in X^{++}$.
In~particular, there holds
\begin{equation*}
\beta(\omega, \mathbf{e}_i) \mathbf{e} \le S(\omega) \mathbf{e}_{i}
\le \varkappa(\omega) \beta(\omega, \mathbf{e}_i) \mathbf{e}.
\end{equation*}
Note that there are positive reals $\underline{\delta} \le
\overline{\delta}$ such that
\begin{equation*}
\underline{\delta} (1,1,\dots,1)^{\top} \le \mathbf{e} \le
\overline{\delta} (1,1,\dots,1)^{\top},
\end{equation*}
consequently
\begin{equation*}
\underline{\delta} \beta(\omega, \mathbf{e}_i)(1,1,\dots,1)^{\top}
\le S(\omega) \mathbf{e}_i \le \overline{\delta} \varkappa(\omega)
\beta(\omega, \mathbf{e}_i)(1,1,\dots,1)^{\top}.
\end{equation*}
Since
\begin{equation*}
S(\omega) \mathbf{e}_i =
(s_{1i}(\omega),s_{2i}(\omega),\dots,s_{Ni}(\omega))^{\top},
\end{equation*}
one has
\begin{equation*}
\beta(\omega,\mathbf{e}_i) \le
\frac{m_{c,i}(\omega)}{\underline{\delta}}, \quad \varkappa(\omega)
\beta(\omega,\mathbf{e}_i) \ge
\frac{M_{c,i}(\omega)}{\overline{\delta}},
\end{equation*}
hence
\begin{equation*}
\frac{M_{c,i}(\omega)}{m_{c,i}(\omega)} \le
\frac{\overline{\delta}}{\underline{\delta}}\varkappa(\omega) \quad
\forall i = 1,2,\dots,N.
\end{equation*}
It then follows that $\lnplus(\ln{M_{c,i}} - \ln{m_{c,i}}) \in
L_1(\OFP)$ for $i = 1, 2,\dots, N$.

Conversely, suppose that $s_{ij}(\omega) > 0$ for $i, j = 1, \dots,
N$ and $\lnplus(\ln{M_{c,i}} - \ln{m_{c,i}}) \in L_1(\OFP)$ for $i =
1, \dots, N$. Note that for any $u = (u_1, \dots, \allowbreak
u_n)^{\top} \in X^{+} \setminus \{0\}$, $u = u_1 \mathbf{e}_1  +
\dots + u_N \mathbf{e}_N$. Hence
\begin{equation*}
u_1 m_{c,1}(\omega) + \dots + u_N m_{c,N}(\omega) \le
\big(S(\omega)u\big)_i \le u_1 M_{c,1}(\omega) + \dots + u_N
M_{c,N}(\omega),\ 1 \le i \le N.
\end{equation*}
Put $\mathbf{e} = (1, \dots, 1)^{\top}/\sqrt{N}$.  Let
\begin{equation*}
\beta(\omega,u) = \sqrt{N} \left(u_1 m_{c,1}(\omega) + \dots + u_N
m_{c,N}(\omega)\right),
\end{equation*}
and
\begin{equation*}
\varkappa(\omega) = N \max_{1 \le i \le N}
\frac{M_{c,i}(\omega)}{m_{c,i}(\omega)}.
\end{equation*}
Then $\varkappa$ is measurable and
\begin{equation*}
\beta(\omega,u) \mathbf{e} \le S(\omega)u \le \varkappa(\omega)
\beta(\omega,u) \mathbf{e}
\end{equation*}
and
\begin{equation*}
\lnplus{\ln{\varkappa(\omega)}} \le \ln{N} + \max\limits_{1 \le i \le
N} \lnplus(\ln{M_{c,i}(\omega)} - \ln{m_{c,i}(\omega)}).
\end{equation*}
Therefore $\lnplus{\ln{\varkappa}} \in L_1(\OFP)$.

(2) It can be proved by arguments similar to those in the proof of
(1).

(3) First, assume that (B4) holds.

Copying the proof of (1) we obtain that $s_{ij}(\omega) > 0$ for all
$\omega \in \Omega$, $i, j = 1, 2, \dots, N$, and
\begin{equation*}
\frac{M_{c,i}(\omega)}{m_{c,i}(\omega)} \le
\frac{\overline{\delta}}{\underline{\delta}} \varkappa(\omega) \quad
\forall \ \omega \in \Omega,\  i = 1,2,\dots, N.
\end{equation*}
Hence
\begin{equation*}
\ln{M_{c,i}(\omega)} - \ln{m_{c,i}(\omega)} \le
\ln{\overline{\delta}} - \ln{\underline{\delta}} +
\ln{\varkappa(\omega)} \quad \forall \ \omega \in \Omega,\ i =
1,2,\dots, N.
\end{equation*}
Therefore, $\ln{M_{c,i}} - \ln{m_{c,i}} \in L_1(\OFP)$.

Similarly, it can be proved that $\ln{M_{r,i}} - \ln{m_{r,i}} \in
L_1(\OFP)$. Therefore (D2)(iii) holds.

Next, we assume that (D2)(iii) holds. Copying the proof of (1) we see
that (B3) holds with $\mathbf{e} = (1, \dots, 1)^{\top}/\sqrt{N}$,
and
\begin{equation*}
\varkappa(\omega) = N \max_{1 \le i \le N}
\frac{M_{c,i}(\omega)}{m_{c,i}(\omega)}
\end{equation*}
Therefore,
\begin{equation*}
\ln{\varkappa(\omega)} \le \ln{N} + \max_{1 \le i \le N}
(\ln{M_{c,i}(\omega)} - \ln{m_{c,i}(\omega)}).
\end{equation*}
It then follows that $\ln{\varkappa} \in L_1(\OFP)$.

Similarly, it can be proved that (B3)$^*$ holds with $\mathbf{e}^{*}
= \mathbf{e}$ and $\ln{\varkappa} \in L_1(\OFP)$.

By $\mathbf{e^*} = \mathbf{e}$, $\langle \mathbf{e}, \mathbf{e^*}
\rangle > 0$. Therefore (B4) holds.
\end{proof}

\begin{proposition} [Strong positivity in one direction]
\label{matrix-prop3}
Assume that $s_{ij}(\omega) > 0$ for all $i, j = 1, 2, \dots, N$ and
$\omega \in \Omega$.
\begin{itemize}
\item[{\rm (1)}]
\textup{(B5)} holds with $\mathbf{\overline{e}} \in X^{++}$ if
and only if \textup{(D3)(i)} holds.
\item[{\rm (2)}]
\textup{(B5)$^*$} holds with $\mathbf{\overline{e}}^{*} \in
X^{++}$ if and only if \textup{(D3)(ii)} holds.
\end{itemize}
\end{proposition}
\begin{proof}
(1) First we assume that (B5) holds with $\bar{\mathbf{e}} \in
X^{++}$. Then $(\mathbf{\overline{e}})_i > 0$ for $i = 1,2,\dots,N$.
Note that there are positive reals $\underline{\delta} \le
\overline{\delta}$ such that
\begin{equation*}
\underline{\delta} (1, \dots, 1)^{\top} \le \mathbf{\overline{e}} \le
\overline{\delta} (1, \dots, 1)^{\top}.
\end{equation*}
There holds
\begin{equation*}
(S(\omega)(1, \dots, 1)^{\top})_{i} \ge \frac{1}{\overline{\delta}}
(S(\omega)\mathbf{\overline{e}})_{i} \ge
\frac{\nu(\omega)}{\overline{\delta}} (\mathbf{\overline{e}})_i \ge
\frac{\underline{\delta} \nu(\omega)}{\overline{\delta}}.
\end{equation*}
Observe that
\begin{equation*}
S(\omega) (1, \dots, 1)^{\top} = \Big(\sum_{j=1}^N s_{1j}(\omega),
\sum_{j=1}^N s_{2j}(\omega),\dots, \sum_{j=1}^N
s_{Nj}(\omega)\Big)^{\top},
\end{equation*}
consequently
\begin{equation*}
\nu(\omega) \le \frac{\underline{\delta}}{\overline{\delta}}
m_r(\omega).
\end{equation*}
This implies that $m_r(\omega) > 0$ for each $\omega \in \Omega$ and
$\ln^{-}{m_r} \in L_1(\OFP)$. Hence (D3)(i) holds.

Conversely, assume that (D3)(i) holds. Put $\mathbf{\overline{e}} =
(1, \dots, 1)^{\top}/\sqrt{N}$. Let
\begin{equation*}
\nu(\omega) = m_r(\omega).
\vspace{-0.05in}\end{equation*}
By  arguments as above,
\begin{equation*}
S(\omega) \mathbf{\overline{e}} \ge \nu(\omega) \mathbf{\overline{e}}
\quad \forall \ \omega \in \Omega.
\end{equation*}
This implies that (B5) holds with $\mathbf{\overline{e}} \in X^{++}$.

(2) It can be proved by similar arguments.
\end{proof}

\begin{remark}
\label{matrix-rm1}
Assume $s_{ij}(\omega) > 0$ for all $i, j = 1, 2, \dots,N$.
\begin{itemize}
\item[{\rm (1)}]
If $\lnplus(\ln{M} - \ln{m}) \in L_1(\OFP)$, then
$\lnplus(\ln{M_{c,i}} - \ln{m_{c,i}}) \in L_1(\OFP)$,
$\lnplus(\ln{M_{r,i}} - \ln{m_{r,i}}) \in L_1(\OFP)$, and hence
\textup{(D2)(i)--(ii)} holds.
\item[{\rm (2)}]
If $\ln{M} - \ln{m} \in L_1(\OFP)$, then $\ln{M_{c,i}} -
\ln{m_{c,i}} \in L_1(\OFP)$, $\ln{M_{r,i}} - \ln{m_{r,i}} \in
L_1(\OFP)$, and hence \textup{(D2)(iii)} holds.
\item[{\rm (3)}]
If $\ln^{-}{m} \in L_1(\OFP)$, then \textup{(D3)(i)--(ii)} holds.
\end{itemize}
\end{remark}

\begin{theorem}
\label{matrix-thm}
\begin{itemize}
\item[{\rm (1)}]
{\em (Entire positive orbits)}  Assume \textup{(D1)(iii)} and
that $S(\omega)(X^{+} \setminus \{0\}) \subset X^{+} \setminus
\{0\}$ for all $\omega \in \Omega$.  Then for any $\omega \in
\Omega$ there is an entire positive orbit $v_{\omega} \colon \ZZ
\to X^{+} \setminus \{0\}$ of $U_{\omega}$.

\item[{\rm (2)}]
{\em (Generalized principal Floquet subspaces and Lyapunov
exponents)} Assume \textup{(D1)(iii)} and \textup{(D2)(i)--(ii)}.
Then $\Phi$ and $\Phi^{*}$ admit families of generalized
principal Floquet subspaces \allowbreak
$\{\tilde{E}_1(\omega)\}_{\omega \in \tilde{\Omega}_1}
\allowbreak = \allowbreak \{\spanned{\{w(\omega)\}}\}_{\omega \in
\tilde{\Omega}_1}$ and $\{\tilde{E}_1^*(\omega)\}_{\omega \in
\tilde{\Omega}_1^*} \allowbreak =
\{\spanned{\{w^*(\omega)\}}\}_{\omega \in \tilde{\Omega}_1^*}$,
with $w(\omega), w^{*}(\omega) \in X^{++}$ for all $\omega \in
\tilde{\Omega}_1$.

\item[{\rm (3)}]
{\em (Generalized exponential separation)} Assume
\textup{(D1)(iii)} and \textup{(D2)(i)--(iii)}. Then there is
$\tilde{\sigma} \in (0,\infty]$ such that the triple
$\{\tilde{E}_1(\cdot), \tilde{F}_1(\cdot), \tilde{\sigma} \}$
generates a generalized exponential separation, where
$\tilde{F}_1(\omega) := \{\, u \in X: \langle u, w^{*}(\omega)
\rangle = 0 \,\}$.  If we assume moreover \textup{(D1)(ii)}, then
$S(\omega) \tilde{F}_1(\omega) = \tilde{F}_1(\theta\omega)$ for
any $\omega \in \tilde{\Omega}_1$.

\item[{\rm (4)}]
{\em (Finiteness of principal Lyapunov exponent)} Assume
\textup{(D1)(iii)}, \textup{(D2)(i)--(iii)}, and \textup{(D3)(i)}
or \textup{(ii)}. Then $\tilde{\lambda}_1 > -\infty$, where
$\tilde{\lambda}_1$ is the generalized principal Lyapunov
exponent associated to $\{\tilde E_1(\omega)\}$.
\end{itemize}
\end{theorem}

\begin{proof}[Proof of Theorem \ref{matrix-thm}]
(1) It follows from Theorem~\ref{thm:entire-orbits-existence}(1).

(2) Parts of it follow from Propositions \ref{matrix-prop1}(1), (3),
\ref{matrix-prop2}(1), (2), Theorem \ref{theorem-w} and its analog
for the dual system.

(3) It follows from Propositions \ref{matrix-prop1}(1), (3),
\ref{matrix-prop2}(1)--(3) and Theorem \ref{separation-thm}.  The
last sentence follows from (D1)(ii).

(4) It follows from Proposition \ref{matrix-prop3} and Theorem
\ref{separation-thm}.
\end{proof}

\begin{remark}
{\em In \cite{AGD} the authors investigated the principal Lyapunov
exponent and principal Floquet subspaces for the random dynamical
system generated by $S(\omega)$ satisfying
\begin{equation}
\label{min-max-eq}
\ln^{-}{m} \in L_1(\OFP), \quad \lnplus{M} \in
L_1(\OFP).
\end{equation}
Theorem \ref{matrix-thm} extends the results in \cite{AGD} in the
following aspects.

\begin{itemize}
\item[{\rm (1)}]
The result in Theorem \ref{matrix-thm}(1) is new.
\item[{\rm (2)}]
The conditions in Theorem \ref{matrix-thm}(2) and (3) are weaker
than the conditions posed in \cite{AGD} and the results in
Theorem \ref{matrix-thm}(2) and (3) are not covered in
\cite{AGD}. Observe that $\tilde{\lambda}_1$ in
Theorem~\ref{matrix-thm}(2) and (3) may be $-\infty$.
\item[{\rm (3)}]
Theorem \ref{matrix-thm}(4) recovers the results in \cite{AGD}.
\end{itemize}}
\end{remark}

\begin{remark}
\label{grow-along-principal-direction-rk1}
{\em Theorem \ref{matrix-thm}(3) implies that for any $u_0 \in X^{+}
\setminus \{0\}$ and $\omega \in \tilde{\Omega}_1$,
\begin{equation*}
\lim_{n\to\infty} \frac{1}{n}\ln\norm{S^n(\omega)u_0} =
\tilde{\lambda}_1
\end{equation*}
and
\begin{equation*}
\limsup_{n\to\infty} \frac{1}{n} \ln{\left\lVert
\frac{S^n(\omega)u_0}{\norm{S^n(\omega)u_0}} - w(\theta^{n}\omega)
\right\rVert} \le -\tilde{\sigma}.
\end{equation*}
Hence $S^n(\omega)u_0$ decreases or increases exponentially at the
rate $\tilde{\lambda}_1$, and its direction converges exponentially
at the rate at~least $\tilde{\sigma}$ toward the direction of
$w(\theta^n\omega)$.}
\end{remark}

\begin{example}
{\em (Random Leslie matrices). Assume that, for each $\omega \in
\Omega$, $S(\omega)$ is a random Leslie matrix, i.e.,
\begin{equation*}
S(\omega) =
\left(\begin{matrix} m_1(\omega) & m_2(\omega) &
m_3(\omega) & \cdots &  m_{N-1}(\omega) & m_N(\omega) \\
b_1(\omega) & 0 & 0 & \cdots & 0 & 0 \\
0 & b_2(\omega) & 0 & \cdots & 0 & 0 \\
0 & 0 & b_3(\omega) & \cdots & 0 & 0 \\
\vdots & \vdots & \vdots & \ddots & \vdots & \vdots \\
0 & 0 & 0 & \cdots & b_{N-1}(\omega)& 0
\end{matrix}\right),
\end{equation*}
where $m_j$ ($ > 0$) ($j = 1, 2, \dots, N$) represent the number of
offspring an individual in age group $j$ in the current time
contributes to the first age group at next time, $b_j$ ($ > 0$) ($j =
1, 2, \dots, N-1$) denote the proportion of individuals of age $j$ in
the current time surviving to age $j+1$ at next time.}
\end{example}

We have
\begin{theorem}
\label{leslie-thm}
Assume that $\ln^{-}{m_0}, \ \lnplus{M_0} \in L_1(\OFP)$, where
\begin{equation*}
m_0(\omega) = \min\{m_1(\omega), m_2(\omega), \dots, m_N(\omega),
b_1(\omega), \dots, b_{N-1}(\omega)\}
\end{equation*}
and
\begin{equation*}
M_0(\omega) = \max\{m_1(\omega), m_2(\omega), \dots, m_N(\omega),
b_1(\omega), \dots, b_{N-1}(\omega)\}.
\end{equation*}
Then \textup{(D1)(i)--(iii)}, \textup{(D2)(i)--(iii)} and
\textup{(D3)(i)--(ii)} are satisfied with $S(\omega)$ replaced by
$S^{(N)}(\omega)$.  Consequently, there are
\begin{itemize}
\item[$\diamond$]
$\tilde{\Omega}_1 \subset \Omega$ with $\PP(\tilde{\Omega}_1) =
1$ and $\theta(\tilde{\Omega}_1) = \tilde{\Omega}_1$,
\item[$\diamond$]
$\{\tilde{E}_1(\omega)\}_{\omega \in \tilde{\Omega}_1} =
\{\spanned{\{w(\omega)\}}\}_{\omega \in \tilde{\Omega}_1}$ with
$w \colon \tilde{\Omega}_1 \to X^{++}$ being $(\mathfrak{F},
\allowbreak \mathfrak{B}(X))$-measurable,
\item[$\diamond$]
$\{\tilde{F}_1(\omega)\}_{\omega \in \tilde{\Omega}_1}$ with
$\tilde{F}_1(\omega) = \{\, u \in X: \langle u, w^{*}(\omega)
\rangle = 0 \,\}$ for each $\omega \in \tilde{\Omega}_1$, $w^{*}
\colon \tilde{\Omega}_1 \to X^{++}$ being $(\mathfrak{F},
\allowbreak \mathfrak{B}(X))$\nobreakdash-\hspace{0pt}measurable,
\item[$\diamond$]
$\tilde{\lambda}_1 \in (-\infty,\infty)$, and $\tilde{\sigma} \in
(0,\infty]$,
\end{itemize}
such that
\begin{itemize}
\item
$\{\tilde{E}_1(\omega)\}_{\omega \in \tilde{\Omega}_1}$ is a
family of generalized principal Floquet subspaces of $S(\cdot)$
and $\tilde{\lambda}_1$ is the generalized principal Lyapunov
exponent associated to $\{\tilde{E}_1(\omega)\}_{\omega \in
\tilde{\Omega}_1}$;

\item
$\{\tilde{E}_1(\omega), \tilde{F}_1(\omega),\tilde{\sigma}\}$
generates a generalized exponential separation of $S(\cdot)$;

\item
$S(\omega) \tilde{E}_1(\omega) = \tilde{E}_1(\theta\omega)$ and
$S(\omega) \tilde{F}_1(\omega) = \tilde{F}_1(\theta\omega)$ for
each $\omega \in \tilde{\Omega}_1$.
\end{itemize}
\end{theorem}
\begin{proof}[Indication of proof]
The theorem follows from Theorem~\ref{matrix-thm} and
Remark~\ref{assumption-rk}.
\end{proof}

\section{Random Cooperative and Type-$K$ Monotone Systems of
Ordinary Differential Equations}
\label{section-cooperative}

In this section we consider applications of the general results
stated in  Section~\ref{general-theory} to random dynamical systems
generated by linear random cooperative systems and linear random
type-$K$ monotone systems.

It is a standing assumption in
Subsections~\ref{subsection-cooperative}
and~\ref{subsection-competitive} that $(\theta_{t})_{t \in \RR}$ is
an ergodic metric dynamical system on $\OFP$, where the probability
measure $\PP$ is complete.

Further, in Subsections~\ref{subsection-cooperative}
and~\ref{subsection-competitive} we make the following standing
assumption.

\medskip
\noindent\textbf{(OP0)} (Measurability) {\em $C \colon \Omega \to
\RR^{N \times N}$, where $C$ stands for $A$ in
Subsection~\ref{subsection-cooperative} and for $B$ in
Subsection~\ref{subsection-competitive}, is $(\mathfrak{F},
\mathfrak{B}(\RR^{N \times N}))$\nobreakdash-\hspace{0pt}measurable,
and $[\, \RR \ni t \mapsto \norm{C(\theta_{t}\omega)} \allowbreak \in
\RR\,]$ belongs to $L_{1,\mathrm{loc}}(\RR)$ for all $\omega \in
\Omega$.}
\medskip

We remark here that (OP0) is satisfied if $C \in L_1(\OFP, \RR^{N
\times N})$ (see~\cite[Example 2.2.8]{Arn}).

\medskip
Under the assumption (OP0), for any $\omega \in \Omega$ and $u_0 \in
X$ there exists a unique solution $[\, \RR \ni t \mapsto u(t; \omega,
u_0) \in X \,]$ of the linear system
\begin{equation}
\label{ODE-system}
u' = C(\theta_{t}\omega) u
\end{equation}
satisfying the initial condition
\begin{equation}
\label{ODE-system-initial-condition}
u(0; \omega, u_0) = u_0.
\end{equation}
The solution is understood in the Carath\'eodory sense: The function
$[\, t \mapsto u(t; \omega, u_0) \,]$ is absolutely continuous, $u(0;
\omega, u_0) = u_0$, and for Lebesgue\nobreakdash-\hspace{0pt}a.e. $t
\in \RR$ there holds
\begin{equation*}
\frac{du(t; \omega, u_0)}{dt} = C(\theta_{t}\omega) u(t; \omega,
u_0).
\end{equation*}
Define
\begin{equation*}
U_{\omega}(t)u_0 := u(t; \omega, u_0), \qquad (t \in \RR,\ \omega \in
\Omega, \ u_0 \in X).
\end{equation*}
$((U_{\omega}(t))_{\omega \in \Omega, t \in [0,\infty)},
(\theta_{t})_{t \in \RR})$ is a measurable linear
skew\nobreakdash-\hspace{0pt}product semiflow on $X$ covering
$\theta$ (for proofs see, e.g, \cite[Section 2.2]{Arn}).  Indeed, all
the relations in the definition are satisfied for $t \in \RR$, and we
can legitimately call $\Phi = ((U_{\omega}(t))_{\omega \in \Omega, t
\in \RR}, (\theta_{t})_{t \in \RR})$ a {\em measurable linear
skew\nobreakdash-\hspace{0pt}product flow\/} on $X$ generated by
\eqref{ODE-system}.  In~particular, (B1)(ii) holds.

\subsection{Random Cooperative Systems of Ordinary Differential
Equations}
\label{subsection-cooperative}
In this subsection, we consider applications of the general results
stated in  Section~\ref{general-theory} to the following cooperative
system of ordinary differential equations
\begin{equation}
\label{cooperative-system}
\dot u(t) = A(\theta_{t}\omega) u(t), \qquad \omega \in \Omega, \ t
\in \RR, \ u \in \RR^N,
\end{equation}
where
\begin{equation*}
A(\omega) =
\left(\begin{matrix}
a_{11}(\omega)&a_{12}(\omega)&\cdots&a_{1N}(\omega) \\
a_{21}(\omega)&a_{22}(\omega)&\cdots&a_{2N}(\omega) \\
\vdots & \vdots & \ddots & \vdots \\
a_{N1}(\omega)&a_{N2}(\omega)&\cdots&a_{NN}(\omega)
\end{matrix}\right).
\end{equation*}
All the notations of $X$, $X^{+}$, etc., are as in
Section~\ref{section-matrix}.  Recall that $\lVert \cdot \rVert_{1}$
is the $\ell_1$\nobreakdash-\hspace{0pt}norm in $X$, $\lVert u
\rVert_{1} = \abs{u_1} + \dots + \abs{u_N}$ for $u = (u_1, \dots,
u_N)^{\top}$.

In the rest of this subsection we assume that (OP0) holds. We state
the standing assumptions on $A(\omega)$.

Let
\begin{equation}
\label{cooperative-aux-eq}
\tilde{a}_{ii}(\omega) := \min_{0 \le t \le 1} \int_0^{t}
a_{ii}(\theta_\tau\omega) \, d\tau, \qquad \bar{a}_{ij}(\omega) :=
\min_{0 \le s \le 1}\int_{s}^1 a_{ij}(\theta_\tau\omega) \, d\tau.
\end{equation}

\noindent \textbf{(O1)} (Cooperativity) {\em $a_{ij}(\omega) \ge 0$
for all $i \ne j$, $i, j = 1, 2,\dots, N$ and $\omega \in \Omega$.}

\medskip
\noindent \textbf{(O2)} (Integrability) {\em The function $[\, \Omega
\ni \omega \mapsto \max\limits_{1 \le i,j \le N} a_{ij}(\omega) \,]
\,]$ is in \allowbreak $L_1(\OFP)$.}

\medskip
\noindent \textbf{(O3)} (Irreducibility) {\em There is an
$(\mathfrak{F},
\mathfrak{B}(\RR))$\nobreakdash-\hspace{0pt}measurable function
$\delta \colon \Omega \to (0,\infty)$  such that for each $\omega \in
\Omega$ and $i \in \{1, 2, \dots,N\}$ there are $j_{1} = i, j_2, j_3,
\dots, j_N \in \{1, 2, \dots, N\}$ satisfying
\begin{itemize}
\item[{\rm (i)}]
$\{j_1, j_2, \dots, j_{N}\} = \{1, 2,\dots, N\}$ and $a_{j_{l}
j_{l+1}}(\theta_{t}\omega) \ge \delta(\omega)$ for
Lebesgue\nobreakdash-\hspace{0pt}a.e.\ $0 \le t \le 1$ and $l =
1, 2, \dots, N-1$.
\item[{\rm (ii)}]
$\lnplus{\ln \bigl(\overline{\beta}/\underline{\beta}\bigr)} \in
L_1(\OFP)$, where
\begin{equation*}
\overline{\beta}(\omega) = \exp{\biggl( \int\limits_{0}^{1}
\Bigl( \sum_{l=1}^{N} \max_{1 \le j \le N}
a_{lj}(\theta_{\tau}\omega) \Bigr) \, d\tau \biggr)}, \quad
\underline{\beta}(\omega) = \min_{1 \le i \le N}\beta_i(\omega),
\end{equation*}
and
\begin{align*}
\beta_{i}(\omega) =
\min\Bigl\{&\exp{(\tilde{a}_{j_1j_1}(\omega))},
\exp{(\tilde{a}_{j_{1} j_{1}}(\omega) + \bar{a}_{j_{2}
j_{2}}(\omega))}\delta(\omega), \\
& \exp{(\tilde{a}_{j_{1} j_{1}}(\omega) + \bar{a}_{j_{2}
j_{2}}(\omega) + \bar{a}_{j_{3}
j_{3}}(\omega))}\frac{\delta^2(\omega)}{2!}, \dots,
\\
& \exp{(\tilde{a}_{j_{1} j_{1}}(\omega) + \bar{a}_{j_{2}
j_{2}}(\omega) + \bar{a}_{j_{3} j_{3}}(\omega) + \dots +
\bar{a}_{j_{N} j_{N}}(\omega))}
\frac{\delta^N(\omega)}{N!}\Bigr\}
\end{align*}
with $j_1 = i$.

\item[{\rm (iii)}]
$\ln{\bigl(\overline{\beta}/\underline{\beta}\bigr)} \in
L_1(\OFP)$, where $\overline{\beta}$ and $\underline{\beta}$ are
as in (ii).

\item[{\rm (iv)}]
$\ln^{-}{\underline{\beta}} \in L_1(\OFP)$, where
$\underline{\beta}$ is as in (ii).
\end{itemize}
}

\noindent\textbf{(O3)$^{\prime}$}
(Off\nobreakdash-\hspace{0pt}diagonal positivity)  {\em
\begin{itemize}
\item[{\rm (i)}]
There is an $(\mathfrak{F},
\mathfrak{B}(\RR))$\nobreakdash-\hspace{0pt}measurable function
$\tilde{\delta} \colon \Omega \to (0,\infty)$ such that for any
$\omega \in \Omega$ and any $i \ne j$ there holds
$a_{ij}(\theta_{t}\omega) \ge \tilde{\delta}(\omega)$ for
Lebesgue\nobreakdash-\hspace{0pt}a.e.\ $t \in [0,1]$.

\item[{\rm (ii)}]
$\ln^{+}\ln\big(\overline{\beta}/\underline{\tilde{\beta}} \big)
\in L_1(\OFP)$, where
\begin{equation*}
\overline{\beta}(\omega) = \exp{\biggl( \int\limits_{0}^{1}
\Bigl( \sum_{l=1}^{N} \max_{1 \le j \le N}
a_{lj}(\theta_{\tau}\omega) \Bigr) \, d\tau \biggr)}, \quad
\underline{\tilde{\beta}}(\omega) = \min_{1 \le i \le
N}\tilde{\beta}_i(\omega),
\end{equation*}
and
\begin{equation*}
\tilde{\beta}_{i}(\omega) =
\min\Bigl\{\exp(\tilde{a}_{ii}(\omega)), \bigl( \min_{\substack
{1 \le j \le N \\ j \ne i}} \exp(\tilde{a}_{ii}(\omega) +
\bar{a}_{ij}(\omega)) \bigr)\delta(\omega)\Bigr\}.
\end{equation*}

\item[{\rm (iii)}]
$\ln\big(\overline{\beta}/\underline{\tilde{\beta}}\big) \in
L_1(\OFP)$, where $\overline{\beta}$ and
$\underline{\tilde{\beta}}$ are as in (ii).

\item[{\rm (iv)}]
$\ln^{-} \underline{\tilde{\beta}} \in L_1(\OFP)$, where
$\underline{\tilde{\beta}}$ is as in (ii).
\end{itemize}
}

\begin{proposition} [Positivity]
\label{cooperative-prop1}
Assume \textup{(O1)}.  Then $\Phi$ satisfies \textup{(B2)}.
\end{proposition}
\begin{proof}
See \cite[Thm.~1]{W}.
\end{proof}

\begin{proposition}[Integrability]
\label{cooperative-prop2}
Assume  \textup{(O1)} and \textup{(O2)}.  Then \textup{(B1)(i)} is
satisfied.
\end{proposition}
\begin{proof}
We claim that for any $\omega \in \Omega$ and any $u_0 \in X^{+}$
there holds
\begin{equation}
\label{l1-estimate}
\normone{U_\omega(t)u_0} = \normone{u(t; \omega, u_0)} \le
\exp\biggl(\int\limits_{0}^{t} \Bigl( \sum_{i=1}^{N} \max_{1 \le j
\le N} a_{ij}(\theta_{\tau}\omega) \Bigr) \, d\tau\biggr)
\normone{u_0}
\end{equation}
for all $t \ge 0$.  Indeed, fix $\omega \in \Omega$ and $u_0 \in
X^{+}$ with $\lVert u_0 \rVert_{1} = 1$, and denote $u(\cdot) =
(u_1(\cdot), \dots, u_N(\cdot)) := u(\cdot; \omega, u_0)$.  For each
$1 \le i \le N$ we estimate
\begin{equation*}
\frac{d u_{i}(t)}{dt} = \sum_{j=1}^{N} a_{ij}(\theta_{t}\omega)
u_{j}(t) \le \max_{1 \le j \le N} a_{ij}(\theta_{t}\omega) \cdot
\sum_{k=1}^{N} u_{k}(t),
\end{equation*}
consequently, in view of Proposition~\ref{cooperative-prop1},
\begin{equation*}
\frac{d}{dt} \normone{u(t)} \le \sum_{i=1}^{N} \max_{1 \le j \le N}
a_{ij}(\theta_{t}\omega) \cdot \normone{u(t)},
\end{equation*}
for Lebesgue\nobreakdash-\hspace{0pt}a.e.\ $t \in \RR$.  The
estimate~\eqref{l1-estimate} follows by comparison theorems for
Carath\'eodory solutions, see, e.g., \cite[Thm.~1.10.1]{La-Lee}.

Since, by remarks in Example~\ref{finite-dim-example-standard}, any
$u_0 \in X$ can be written as $u_0^{+} - u_0^{-}$ with
$\normone{u_0^{+}} \le \normone{u_0}$, $\normone{u_0^{-}} \le
\normone{u_0}$, and, moreover, $\normone{u_0} = \normone{u_0^{+}} +
\normone{u_0^{-}}$, we have that
\begin{equation}
\label{cooperative-estimate}
\normone{U_\omega(t)u_0} \le \normone{U_\omega(t)u_0^{+}} +
\normone{U_\omega(t)u_0^{-}} \le \exp\biggl(\int\limits_{0}^{t}
\Bigl( \sum_{i=1}^{N} \max_{1 \le j \le N}
a_{ij}(\theta_{\tau}\omega) \Bigr) \, d\tau\biggr) \normone{u_0}
\end{equation}
for all $\omega \in \Omega$, $u_0 \in X$ and $t \ge 0$.  As the
integrand in the rightmost term in~\eqref{cooperative-estimate} is
nonnegative, we infer that
\begin{equation*}
\lnplus{\lVert U_{\omega}(t) \rVert_{1}} \le \int\limits_{0}^{t}
\Bigl( \sum_{i=1}^{N} \max_{1 \le j \le N}
a_{ij}(\theta_{\tau}\omega) \Bigr) \, d\tau
\end{equation*}
for all $\omega \in \Omega$ and $t \ge 0$.  Since the norms
$\norm{\cdot}$ and $\lVert \cdot \rVert_{1}$ are equivalent, the
assertion follows.
\end{proof}

\begin{proposition}[Focusing]
\label{cooperative-prop3}
\begin{itemize}
\item[{\rm (1)}]
Assume \textup{(O1)} and \textup{(O3)(i)--(ii)}, or \textup{(O1)}
and \textup{(O3)$^{'}$(i)--(ii)}. Then \textup{(B3)},
\textup{(B3)$^*$} hold.
\item[{\rm (2)}]
Assume \textup{(O1)} and \textup{(O3)(i), (iii)}, or
\textup{(O1)} and \textup{(O3)$^{'}$(i), (iii)}. Then
\textup{(B4)} holds.
\end{itemize}
\end{proposition}
\begin{proof}
Observe that, since in (B3) or (B4) the suitable properties are to
hold for $U_{\omega}(1)$ and $U^{*}_{\omega}(1)$ only, we can (and
do) apply the results in Section~\ref{section-matrix}.

First, assume (O1) and (O3)(i)--(ii). Fix $\omega \in \Omega$ and $i
\in \{1, \dots, N\}$.  By (O3)(i), there are $j_1, j_2, \dots, j_N$
with $j_1 = i$ such that $\{j_1, j_2, \dots, j_N\} = \{1, 2,
\allowbreak \dots, \allowbreak N\}$ and
\begin{equation}
\label{cooperative-eq5}
a_{j_{k} j_{k+1}}(\theta_t\omega) \ge \delta(\omega) \quad \text{for}
\quad t \in [0,1].
\end{equation}
Note that
\begin{equation*}
\frac{d u_i(t;\omega,\mathbf{e}_i)}{dt} \ge a_{ii}(\theta_t\omega)
u_i(t;\omega,\mathbf{e}_i)
\end{equation*}
for Lebesgue\nobreakdash-\hspace{0pt}a.e.\ $t \in \RR$ and
$u_i(0;\omega,\mathbf{e}_i) = 1$, so, by comparison results for
Carath\'eodory solutions (\cite[Thm.~1.10.1]{La-Lee}),
\begin{equation}
\label{cooperative-eq6}
u_i(t;\omega,\mathbf{e}_i) \ge \exp{\Bigl( \int\limits_{0}^{t}
a_{ii}(\theta_\tau\omega) \, d\tau\Bigr)}
\end{equation}
for $t > 0$.

Similarly,
\begin{equation*}
\frac{d u_{j_l}(t;\omega,\mathbf{e}_i)}{dt} \ge a_{j_{l}
j_{l}}(\theta_t\omega)u_{j_l}(t;\omega,\mathbf{e}_i) + a_{j_{l-1}
j_{l}}(\theta_t\omega)u_{j_{l-1}}(t;\omega,\mathbf{e}_i), \quad l =
2, 3, \dots, N,
\end{equation*}
for Lebesgue\nobreakdash-\hspace{0pt}a.e.\ $t \in \RR$ and
$u_{j_l}(0;\omega,\mathbf{e}_i) = 0$, so, by comparison results for
Carath\'eodory solutions,
\begin{equation}
\label{cooperative-eq7}
u_{j_l}(t;\omega,\mathbf{e}_i) \ge \int\limits_{0}^{t} \exp\Bigl(
\int\limits_{s}^{t} a_{j_{l} j_{l}}(\theta_{\tau} \omega) \, d\tau
\Bigr) a_{j_{l-1} j_{l}}(\theta_s\omega) \,
u_{j_{l-1}}(s;\omega,\mathbf{e}_i) \, ds
\end{equation}
for $t > 0$ and  $l = 2, 3, \dots, N$.

By \eqref{cooperative-eq6} and \eqref{cooperative-eq7}, there holds
\begin{equation}
\label{cooperative-eq8}
\left\{
\begin{aligned}
& u_{ii}(\omega) & \ge {} & \exp(\tilde{a}_{j_1 j_1}(\omega))
\\
& u_{j_{2} i}(\omega) & \ge {} & \exp(\tilde{a}_{j_1 j_1}(\omega) +
\bar{a}_{j_2 j_2}(\omega)) \delta(\omega)
\\
& u_{j_{3} i}(\omega) & \ge {} & \exp(\tilde{a}_{j_1 j_1}(\omega) +
\bar{a}_{j_2 j_2}(\omega) + \bar{a}_{j_3j_3}(\omega))
\frac{\delta^2(\omega)}{2!} \\
& \vdots &  & \qquad \vdots \\
& u_{j_{N} i}(\omega) & \ge {} & \exp(\tilde{a}_{j_1 j_1}(\omega) +
\bar{a}_{j_2 j_2}(\omega) + \bar{a}_{j_3 j_3}(\omega) + \dots +
\bar{a}_{j_{N} j_{N}}(\omega)) \frac{\delta^N(\omega)}{N!}
\end{aligned}
\right.
\end{equation}
where $\tilde{a}_{ii}(\omega)$ and $\bar{a}_{ii}(\omega)$ are as in
\eqref{cooperative-aux-eq}, and $u_{kl}(\omega) = u_{k}(1; \omega,
\mathbf{e}_{l})$ for $k, l \in \{1, \dots, N\}$.

Further we have
\begin{equation*}
u_{ji}(\omega)  \le \lVert u(1; \omega, \mathbf{e}_i) \rVert_{1} \le
\exp{\biggl( \int\limits_{0}^{1} \Bigl( \sum_{l=1}^{N} \max_{1 \le k
\le N} a_{lk}(\theta_{\tau}\omega) \Bigr) \, d\tau \biggr)} \,\, {\rm
for}\,\, 1\leq j\leq N \quad \text{(by \eqref{l1-estimate})}.
\end{equation*}

We have proved that
\begin{equation}
\label{aux-eq1}
\underline{\beta}(\omega) \le u_{ij}(\omega) \le
\overline{\beta}(\omega)
\end{equation}
for all $\omega \in \Omega$ and $i, j \in \{1, 2, \dots, N\}$, where
$\underline{\beta}(\omega)$ and $\overline{\beta}(\omega)$ are as in
(O3).  Consequently, (D2)(i)--(ii) hold with
\begin{equation*}
\underline{\beta}(\omega) \le m_{c,i}(\omega) \le M_{c,i}(\omega) \le
\overline{\beta}(\omega), \quad \underline{\beta}(\omega) \le
m_{r,i}(\omega) \le M_{r,i}(\omega) \le
\overline{\beta}(\omega)
\end{equation*}
for all $\omega \in \Omega$, $i = 1, 2, \dots, N$, which gives with
the help of Proposition~\ref{matrix-prop2}(1)--(2) that (B3) and
(B3)$^{*}$ hold.

Next, assume (O1) and (O3)$^{'}$(i)--(ii).  Fix $\omega \in \Omega$
and $i \in \{1, \dots, N\}$.  As above,  we estimate
\begin{equation}
\label{cooperative-eq6A}
u_i(t;\omega,\mathbf{e}_i) \ge \exp{\Bigl( \int\limits_{0}^{t}
a_{ii}(\theta_\tau\omega) \, d\tau\Bigr)}, \qquad t \ge 0.
\end{equation}
For $j \in \{1, \dots, N\}$, $j \ne i$, we have
\begin{equation*}
\frac{d u_{j}(t;\omega,\mathbf{e}_i)}{dt} \ge
a_{jj}(\theta_t\omega)u_{j}(t;\omega,\mathbf{e}_i) +
a_{ji}(\theta_t\omega)u_{i}(t;\omega,\mathbf{e}_i)
\end{equation*}
for Lebesgue\nobreakdash-\hspace{0pt}a.e.\ $t \in \RR$ and
$u_{j}(0;\omega,\mathbf{e}_i) = 0$, which gives
\begin{equation}
\label{cooperative-eq7A}
u_{j}(t;\omega,\mathbf{e}_i) \ge \int\limits_{0}^{t} \exp\Bigl(
\int\limits_{s}^{t} a_{jj}(\theta_{\tau}\omega) \, d\tau \Bigr)
a_{ji}(\theta_s\omega) \, u_{i}(s;\omega,\mathbf{e}_i) \, ds
\end{equation}
for $t \ge 0$.  By \eqref{cooperative-eq6A} and
\eqref{cooperative-eq7A}, there holds
\begin{equation}
\label{cooperative-eq8A}
u_{ii}(\omega) \ge \exp(\tilde{a}_{ii}(\omega)) \quad \text{and}
\quad u_{ji}(\omega) \ge \exp(\tilde{a}_{ii}(\omega) +
\bar{a}_{ji}(\omega)) \delta(\omega), \ j \ne i,
\end{equation}
from which it follows that
\begin{equation}
\label{aux-eq2}
\underline{\tilde{\beta}}(\omega) \le u_{ij}(\omega) \le
\overline{\beta}(\omega)
\end{equation}
for all $\omega \in \Omega$ and $i, j \in \{1, 2, \dots, N\}$, where
$\underline{\tilde{\beta}}(\omega)$ and $\overline{\beta}(\omega)$
are as in (O3)$^{'}$.  The rest goes along the lines of the proof in
the above case.

This completes the proof of (1).  The proof of (2) goes in a similar
way.
\end{proof}

\begin{proposition} [Strong positivity in one direction]
\label{cooperative-prop4}
Assume \textup{(O1)} and \textup{(O3)(i), (iv)}, or \textup{(O1)} and
\textup{(O3)$^{'}$(i), (iv)}. Then \textup{(B5)} and
\textup{(B5)$^*$} hold.
\end{proposition}

\begin{proof}
As in the proof of Proposition~\ref{cooperative-prop3} we use
\eqref{aux-eq1} (or \eqref{aux-eq2}) to show that (D3)(i)--(ii) hold
with $m_r(\omega) \ge N \underline{\beta}(\omega)$ and $m_r(\omega)
\ge N \underline{\beta}(\omega)$ (or with $m_r(\omega) \ge
\underline{\tilde{\beta}}(\omega)$ and $m_c(\omega) \ge
\underline{\tilde{\beta}}(\omega)$), and apply
Proposition~\ref{matrix-prop3}.
\end{proof}

For $\omega \in \Omega$, by an {\em entire positive solution\/} of
\eqref{cooperative-system}$_{\omega}$ we understand an entire
positive orbit of $U_{\omega}$, that is, a function $v_{\omega}
\colon \RR \to X^{+}$ such that
\begin{equation*}
U_{\omega}(s) (v_{\omega}(t)) = v_{\omega}(t + s) \qquad \text{for
all } t \in \RR, \ s \in [0, \infty).
\end{equation*}
An entire positive solution $v_{\omega}$ is {\em nontrivial\/} if
$v_{\omega}(t) \in X^{+} \setminus \{0\}$ for each $t \in \RR$.

\begin{theorem}
\label{cooperative-thm}
\begin{itemize}
\item[{\rm (1)}]
{\em (Entire positive solution)} Assume \textup{(O1)}. For any
$\omega \in \Omega$ there exists a nontrivial entire positive
solution of \eqref{cooperative-system}$_{\omega}$.
\item[{\rm (2)}]
{\em (Generalized principal Floquet subspaces and Lyapunov
exponent)} Let \textup{(O1)} and \textup{(O2)} be satisfied.
Moreover, assume \textup{(B3)} and \textup{(B3)$^{*}$}
\textup{(}for~instance, assume \textup{(O3)(i)--(ii)} or
\textup{(O3)$^{'}$ (i)--(ii)}\textup{)}. Then $\Phi$ and
$\Phi^{*}$ admit families of generalized principal Floquet
subspaces $\{\tilde{E}_1(\omega)\}_{\omega \in \tilde{\Omega}_1}
= \{\spanned{\{w(\omega)\}}\}_{\omega \in \tilde{\Omega}_1}$ and
$\{\tilde{E}_1^*(\omega)\}_{\omega \in \tilde{\Omega}_1^*}
\allowbreak = \{\spanned{\{w^*(\omega)\}}\}_{\omega \in
\tilde{\Omega}_1^*}$. Moreover,
\begin{equation}
\label{lambda-eq}
\tilde{\lambda}_1 = \int\limits_{\Omega} \kappa(\omega) \,
d\PP(\omega),
\end{equation}
where $\kappa(\omega) = \langle A(\omega) w(\omega), w(\omega)
\rangle$, and $\tilde\lambda_1$ is the generalized principal
Lyapunov exponent associated to $\{\tilde
E_1(\omega)\}_{\omega\in\tilde\Omega_1}$.
\item[{\rm (3)}]
{\em (Generalized exponential separation)} Let \textup{(O1)} and
\textup{(O2)} be satisfied.  Moreover, assume \textup{(B4)}
\textup{(}for~instance, assume \textup{(O3)(i)} and
\textup{(iii)}, or \textup{(O3)$^{'}$(i)} and
\textup{(iii)}\textup{)}. Then there is $\tilde{\sigma} \in
(0,\infty]$ such that the triple $\{\tilde{E}_1(\omega),
\tilde{F}_1(\omega), \tilde{\sigma} \}$ generates a generalized
exponential separation of $S(\omega)$, where $\tilde{F}_1(\omega)
:= \{\, u \in X: \langle u, w^{*}(\omega) \rangle = 0 \,\}$.
Moreover, $U_{\omega}(t) \tilde{F}_1(\omega) =
\tilde{F}_1(\theta_{t}\omega)$ for any $\omega \in
\tilde{\Omega}_1$;
\item[{\rm (4)}]
{\em (Finiteness of principal Lyapunov exponent)} Let
\textup{(O1)} and \textup{(O2)} be satisfied.  Moreover, assume
\textup{(B5)} or \textup{(B5)$^{*}$} \textup{(}for instance,
assume \textup{(O3)(i), (iii)} and \textup{(iv)}, or
\textup{(O3)$^{'}$(i), (iii)} and \textup{(iv)}\textup{)}.  Then
$\tilde{\lambda}_1 > -\infty$.
\end{itemize}
\end{theorem}
\begin{proof}
(1)  It follows from Theorem~\ref{thm:entire-orbits-existence}(i).

(2)  Parts of it follow from Proposition \ref{cooperative-prop3}(1),
Theorem \ref{theorem-w}, and its counterpart for the dual system. By
Theorem~\ref{theorem-w},
\begin{equation*}
\tilde{\lambda}_1 = \lim\limits_{t \to \infty}
\frac{\ln{\norm{U_{\omega}(t)w(\theta_{t}\omega)}}}{t}
\end{equation*}
for $\PP$\nobreakdash-\hspace{0pt}a.e.\ $\omega \in \Omega$.
Differentiating formally we obtain
\begin{equation*}
\begin{aligned}
&
\left. \frac{d}{dt} \ln{\norm{U_{\omega}(t)w(\theta_{t}\omega)}}
\right\rvert_{t = s} = \left. \frac{1}{2} \frac{d}{dt} \ln{\langle
U_{\omega}(t)w(\theta_{t}\omega), U_{\omega}(t)w(\theta_{t}\omega)
\rangle} \right\rvert_{t = s} \\
& = \frac{\langle
A(\theta_{s}\omega)(U_{\omega}(s)w(\theta_{s}\omega)),
U_{\omega}(s)w(\theta_{s}\omega) \rangle}
{\norm{U_{\omega}(s)w(\theta_{s}\omega)}^2} \\
& = \langle A(\theta_{s}\omega) w(\theta_{s}\omega),
w(\theta_{s}\omega) \rangle = \kappa(\theta_{s}\omega).
\end{aligned}
\end{equation*}
It follows from (OP0) that for each $\omega \in \Omega$ the function
$[\, \RR \ni t \mapsto
A(\theta_{t}\omega)(U_{\omega}(t)w(\theta_{t}\omega)) \in \RR^{N}
\,]$ belongs to $L_{1,\mathrm{loc}}(\RR, \RR^{N})$.  Consequently we
have
\begin{equation*}
\ln{\norm{U_{\omega}(t)w(\theta_{t}\omega)}} = \int\limits_{0}^{t}
\kappa(\theta_{s}\omega) \, ds\quad \forall t\geq 0.
\end{equation*}

For $\omega \in \Omega$ let $w(\omega) = (w_1(\omega), \dots,
w_{N}(\omega))^{\top}$.  We estimate
\begin{equation*}
\begin{aligned}
&
\quad \langle A(\omega) w(\omega), w(\omega) \rangle = \sum_{i=1}^{N}
\Bigl( \sum_{j=1}^{N} a_{ij}(\omega) w_j(\omega) \Bigr) w_i(\omega)
\\
&
\le \Bigl( \max_{1 \le i \le N} \sum_{j=1}^{N} a_{ij}(\omega)
w_j(\omega) \Bigr) \lVert w(\omega) \rVert_1 \le \sqrt{N} \max_{1 \le
i \le N} \sum_{j=1}^{N} a_{ij}(\omega) w_j(\omega) \\
&
\le \sqrt{N} \max_{1 \le i \le N} \Bigl( \bigl( \max_{1 \le j \le N}
a_{ij}(\omega) \bigr) \lVert w(\omega) \rVert_{1} \Bigr) \le N
\sum_{i=1}^{N} \max_{1 \le j \le N} a_{ij}(\omega).
\end{aligned}
\end{equation*}
Hence, by (O1) and (O2), $\kappa^{+} \in L_1(\OFP)$, which allows us
to apply the Birkhoff Ergodic Theorem to get \eqref{lambda-eq}.

(3)  It follows from Proposition \ref{cooperative-prop3}(2) and
Theorem \ref{separation-thm}.  The last sentence follows from
\cite[Theorem~3.8(5)]{MiShPart1}.

(4) It follows from Proposition~\ref{cooperative-prop4} and Theorem
\ref{separation-thm}.
\end{proof}

\begin{remark}
\label{grow-along-principal-direction-rk2}
{\em Theorem \ref{cooperative-thm}(3) implies that for any $u_0 \in
X^{+} \setminus \{0\}$ and $\omega \in \tilde{\Omega}_1$,
\begin{equation*}
\lim_{t\to\infty} \frac{1}{t} \ln\norm{U_{\omega}(t)u_0} =
\tilde\lambda_1
\end{equation*}
and
\begin{equation*}
\limsup_{t\to\infty} \frac{1}{t} \ln{\left\lVert
\frac{U_{\omega}(t)u_0}{\norm{U_{\omega}(t)u_0}} - w(\theta_t\omega)
\right\rVert} \le -\tilde{\sigma}.
\end{equation*}
Hence $U_\omega(t)u_0$ decreases or increases exponentially at the
rate $\tilde{\lambda}_1$, and its direction converges exponentially
at the rate at~least $\tilde{\sigma}$ toward the direction of
$w(\theta_t\omega)$.}
\end{remark}

\subsection{Type-$K$ Monotone Systems of Ordinary Differential
Equations}
\label{subsection-competitive}

In this subsection, we consider applications of the general results
stated in  Section~\ref{general-theory} to the following type-$K$
monotone systems of ordinary differential equations
\begin{equation}
\label{k-competitive-system}
\dot{u}(t) = B(\theta_{t}\omega) u(t), \qquad \omega \in \Omega, \ t
\in \RR, \ u \in \RR^N,
\end{equation}
where
\begin{equation*}
B(\omega) =
\left(\begin{matrix}
b_{11}(\omega)&b_{12}(\omega)&\cdots&b_{1N}(\omega) \\
b_{21}(\omega)&b_{22}(\omega)&\cdots&b_{2N}(\omega) \\
\vdots & \vdots & \ddots & \vdots \\
b_{N1}(\omega)&b_{N2}(\omega)&\cdots&b_{NN}(\omega)
\end{matrix}\right)
\end{equation*}
satisfies the following assumptions.

\medskip
\noindent \textbf{(P1)} (Type-$K$ monotonicity) {\em There are $1 \le
k, l \le N$ such that $k + l = N$, $b_{ij}(\omega) \ge 0$ for $i \ne
j$ and $i, j \in \{1,2,\dots,k\}$ or $i, j \in
\{k+1,k+2,\dots,k+l\}$, and $b_{ij}(\omega) \le 0$ for $i \in
\{1,2,\cdots,k\}$ and $j \in \{k+1,k+2,\cdots,k+l\}$ or $i \in
\{k+1,k+2,\cdots,k+l\}$ and $j \in \{1,2,\cdots,k\}$}

\medskip
\noindent \textbf{(P2)} (Integrability) {\em The function $[\, \Omega
\ni \omega \mapsto \max\limits_{1 \le i,j \le N} \abs{b_{ij}}(\omega)
\,] \,]$ is in \allowbreak $L_1(\OFP)$.}

To simplify the notation, we write $\mathcal{K}$ for $\{1, \dots,
k\}$ and $\mathcal{L}$ for $\{k+1, \dots, n\}$.

Make the following change of variables, $u \mapsto v$, where
\begin{equation*}
v_i =
\begin{cases}
u_i \quad & \text{if} \quad i \in \mathcal{K} \\
-u_i \quad & \text{if} \quad i \in \mathcal{L}.
\end{cases}
\end{equation*}
Then \eqref{k-competitive-system} becomes
\begin{equation}
\label{competitive-cooperative-system}
\dot{v}(t) = A(\theta_{t}\omega)v,
\end{equation}
where $A(\theta_{t}\omega) = (a_{ij}(\theta_{t}\omega))_{i,j=1}^{n}$
with
\begin{equation*}
a_{ij}(\omega)
=\begin{cases} b_{ij}(\omega) & \quad \text{if} \quad i, j \in
\mathcal{K} \text{ or }  i, j \in \mathcal{L} \\
-b_{ij}(\omega) & \quad \text{if} \quad i \in \mathcal{K} \text{ and
} j \in \mathcal{L} \text{ or } i \in \mathcal{L} \text{ and } j \in
\mathcal{K}.
\end{cases}
\end{equation*}
By (P1)--(P2), $A(\omega)$ satisfies (O1)--(O2). Let
\begin{equation*}
\tilde{X}^{+} := \{\, u=(u_1,\dots,u_N)^{\top} :  u_i \ge 0 \text{
for } i \in \mathcal{K} \text{ and } u_i \le 0 \text{ for } i \in
\mathcal{L} \,\}.
\end{equation*}
Then $\tilde{X}^{+}$ is a solid cone in $X$, $(X,\tilde{X}^{+})$
satisfies (B0), and $\Phi$ satisfies (B2) with respect to the order
induced by $\tilde{X}^{+}$.  We say an entire solution $[\, \RR \ni t
\mapsto v(t) \in X \,]$ of \eqref{k-competitive-system} is {\em
positive\/} if $v(t) \in \tilde{X}^{+}$ for any $t \in \RR$.

In the rest of this section, we assume (P1)--(P2) and that
$A(\omega)$ is as in \eqref{competitive-cooperative-system}. The
order in $X$ is referred to the order generated by $\tilde{X}^{+}$.
By Theorem \ref{cooperative-thm} and the relation between
\eqref{k-competitive-system} and
\eqref{competitive-cooperative-system}, we have

\begin{theorem}
\label{k-competitive-thm}
\begin{itemize}
\item[{\rm (1)}]
{\em (Entire positive solution)}  For any $\omega \in \Omega$
there exists a nontrivial entire positive solution of
\eqref{k-competitive-system}$_{\omega}$.
\item[{\rm (2)}]
{\em (Generalized principal Floquet subspaces and Lyapunov
exponent)} Assume  \textup{(B3)} and \textup{(B3)$^{*}$}
\textup{(}for~instance, assume  that $A(\omega)$ satisfies
\textup{(O3)(i)--(ii)} or \textup{(O3)$^{'}$
(i)--(ii)}\textup{)}. Then $\Phi$ and $\Phi^{*}$ admit families
of generalized principal Floquet subspaces
$\{\tilde{E}_1(\omega)\}_{\omega \in \tilde{\Omega}_1} =
\{\spanned{\{w(\omega)\}}\}_{\omega \in \tilde{\Omega}_1}$ and
$\{\tilde{E}_1^*(\omega)\}_{\omega \in \tilde{\Omega}_1^*}
\allowbreak = \{\spanned{\{w^*(\omega)\}}\}_{\omega \in
\tilde{\Omega}_1^*}$. Moreover,
\begin{equation}
\label{lambda-eq-1}
\tilde{\lambda}_1 = \int\limits_{\Omega} \kappa(\omega) \,
d\PP(\omega),
\end{equation}
where $\kappa(\omega) = \langle B(\omega) w(\omega), w(\omega)
\rangle$, and $\tilde\lambda_1$ is the generalized principal
Lyapunov exponent associated to $\{\tilde
E_1(\omega)\}_{\omega\in\tilde\Omega_1}$.
\item[{\rm (3)}]
{\em (Generalized exponential separation)} Assume  \textup{(B4)}
\textup{(}for~instance, assume $A(\omega)$ satisfies
\textup{(O3)(i)} and \textup{(iii)}, or \textup{(O3)$^{'}$(i)}
and \textup{(iii)}\textup{)}. Then there is $\tilde{\sigma} \in
(0,\infty]$ such that the triple $\{\tilde{E}_1(\omega),
\tilde{F}_1(\omega), \tilde{\sigma} \}$ generates a generalized
exponential separation of $S(\omega)$, where $\tilde{F}_1(\omega)
:= \{\, u \in X: \langle u, w^{*}(\omega) \rangle = 0 \,\}$.
Moreover, $U_{\omega}(t) \tilde{F}_1(\omega) =
\tilde{F}_1(\theta_{t}\omega)$ for any $\omega \in
\tilde{\Omega}_1$;
\item[{\rm (4)}]
{\em (Finiteness of principal Lyapunov exponent)} Assume
\textup{(B5)} or \textup{(B5)$^{*}$} \textup{(}for instance,
assume that $A(\omega)$ satisfies \textup{(O3)(i), (iii)} and
\textup{(iv)}, or \textup{(O3)$^{'}$(i), (iii)} and
\textup{(iv)}\textup{)}.  Then $\tilde{\lambda}_1 > -\infty$.
\end{itemize}
\end{theorem}

\subsection{An Example}
\label{subsection:example}
In the present subsection we give an example showing that the
invariant decomposition provided by a generalized exponential
separation may be finer than the Oseledets decomposition given in
Theorem~\ref{Oseledets-thm}.

\medskip
We start with the two\nobreakdash-\hspace{0pt}dimensional torus,
written as $(0,1] \times (0,1]$.  Choose an irrational number $\rho
\in (0,1)$.  Define
\begin{equation*}
\theta_t(\omega_1,\omega_2) := (\omega_1 + t, \omega_2 + {\rho}t),
\quad t \in \RR, \ (\omega_1, \omega_2) \in (0,1] \times (0,1],
\end{equation*}
where addition is understood modulo $1$.

Let $\Omega$ be the set of those $\omega \in (0,1] \times (0,1]$ for
which $\theta_{t}\omega \ne (1,1)$ for any $t \in \RR$.
$\mathfrak{F}$ equals the family of
Lebesgue\nobreakdash-\hspace{0pt}measurable subsets of $\Omega$, and
$\PP$ is the normalized Lebesgue measure on $\Omega$.  $(\OFP,
(\theta_{t})_{t\in\RR})$ is an ergodic metric flow, with complete
$\PP$.

Define an $(\mathfrak{F},
\mathfrak{B}(\RR))$\nobreakdash-\hspace{0pt}measurable function $a
\colon \Omega \to \RR$,
\begin{equation*}
a(\omega) = a(\omega_1, \omega_2) := - \frac{1}{(\omega_1 +
\omega_2)^{2}}.
\end{equation*}

For $\omega \in \Omega$ put
\begin{equation*}
A(\omega) = \left(
\begin{matrix} a(\omega) & 1 \\
1 & a(\omega)
\end{matrix} \right).
\end{equation*}
The first part of (OP0) is satisfied.  Further, for each $\omega \in
\Omega$ the discontinuity points of the function $[\, \RR \ni t
\mapsto A(\theta_{t}\omega) \in \RR^{2 \times 2} \,]$ are precisely
those $t \in \RR$ at which either $(\theta_{t}\omega)_1 = 1$ or
$(\theta_{t}\omega)_2 = 1$ (but not both, thanks to our choice of
$\Omega$).  At any of such points the function is left- or
right\nobreakdash-\hspace{0pt}continuous, with finite limits.
Therefore its is locally bounded, hence locally integrable, and the
second part of (OP0) is satisfied, too.

(O1) and (O2) are  obvious. It follows from
Propositions~\ref{cooperative-prop1} and~\ref{cooperative-prop2} that
(B1)(i) and (B2) hold.

Observe that for each $\omega \in \Omega$ one has
\begin{equation}
\label{example-formula}
U_{\omega}(t) = \exp\Bigl( \int\limits_{0}^{t} a(\theta_{\tau}\omega)
\, d\tau \Bigr) e^{tB}, \qquad t \in \RR,
\end{equation}
where $B = \begin{pmatrix} 0 & 1 \\ 1 & 0 \end{pmatrix}$,
consequently
\begin{equation*}
U_{\omega}(t) = \exp\Bigl( \int\limits_{0}^{t} a(\theta_{\tau}\omega)
\, d\tau \Bigr) \begin{pmatrix} \cosh{t} & \sinh{t} \\ \sinh{t} &
\cosh{t}
\end{pmatrix}, \qquad t \in \RR.
\end{equation*}
It is straightforward that
\begin{equation*}
(\sinh{1}) \mathbf{e} \le e^{B} \mathbf{e}_{i} \le (\cosh{1})
\mathbf{e}, \quad i = 1, 2,
\end{equation*}
which gives that for each $\omega \in \Omega$ and each $u \in X^{+}
\setminus \{0\}$, $u^{*} \in (X^{*})^{+} \setminus \{0\}$, there are
$\beta(\omega, u) > 0$, $\beta^{*}(\omega, u^{*}) > 0$ such that
\begin{gather*}
\beta(\omega, u) \mathbf{e} \le U_{\omega}(t) u  \le (\coth{1})
\beta(\omega, u) \mathbf{e}, \\
\beta^{*}(\omega, u^{*}) \mathbf{e}^{*} \le U^{*}_{\omega}(t) u^{*}
\le (\coth{1}) \beta^{*}(\omega, u^{*}) \mathbf{e}^{*}.
\end{gather*}
Consequently, (B4) holds with $\varkappa$ and $\varkappa^{*}$
constantly equal to $\coth{1}$.

Due to Eq.~\eqref{example-formula}, for each $\omega \in \Omega$ the
subspace $\tilde{E}_1(\omega)$, provided by
Theorem~\ref{cooperative-thm}(3), equals the invariant subspace of
$B$ corresponding to the principal eigenvalue of $B$, that is,
$\tilde{E}_1(\omega) = \spanned\{(1,1)^{\top}\}$, whereas
$\tilde{F}_1(\omega)$ equals the complementary invariant subspace of
$B$, that is, $\tilde{F}_1(\omega) = \spanned\{(1,-1)^{\top}\}$.
In~particular, $w(\omega) = (1,1)^{\top}/\sqrt{2}$ for any $\omega
\in \Omega$.

We apply \eqref{lambda-eq} to calculate $\tilde{\lambda}_1$.  There
holds $\kappa(\omega) = 1 + a(\omega)$ for each $\omega \in \Omega$.
We estimate
\begin{equation*}
\tilde{\lambda}_1 = \int\limits_{\Omega} (1 + a(\omega)) \,
d\PP(\omega) \le 1 - \int\limits_{0}^{1} d\omega_1 \int\limits_{0}^{1
- \omega_1} \frac{d\omega_2}{(\omega_1 + \omega_2)^{2}} = - \infty.
\end{equation*}
Obviously, $\tilde{\lambda}_2 = \tilde{\lambda}_1 = - \infty$.  To
calculate $\tilde{\sigma}$, observe that, by \eqref{example-formula},
\begin{equation*}
\tilde{\sigma} = - \lim_{t\to\infty} \frac{1}{t}
\ln{\frac{\norm{U_{\omega}(t)|_{\tilde{F}_1(\omega)}}}
{\norm{U_{\omega}(t) w(\omega)}}} = - \lim_{t\to\infty} \frac{1}{t}
\ln{\frac{\norm{e^{tB}|_{\spanned\{(1,-1)^{\top}\}}}}
{\norm{e^{tB}|_{\spanned\{(1,1)^{\top}\}}}}} = 2 \quad \forall \omega
\in \Omega.
\end{equation*}

Notice that Case (i) in Theorem~\ref{Oseledets-thm} is satisfied and
no invariant families $\{E_1(\omega)\}$, $\{\hat{F}_1(\omega)\}$, can
be defined in terms of exponential rates of convergence.

\section*{Acknowledgments}
The authors thank the referees for their helpful suggestions.

\end{document}